\documentclass[12pt]{article}

\usepackage{amsmath,amssymb,amsthm,textcomp,mathtools}
\usepackage{amsfonts,graphicx}
\usepackage[mathscr]{eucal}
\usepackage{color}
\usepackage{csquotes}
\usepackage{enumitem}
\usepackage{authblk} 
\usepackage{comment}
\usepackage{appendix}
\numberwithin{equation}{section}

\theoremstyle{definition}

\numberwithin{equation}{section}

\newtheorem{theorem}{\bf Theorem}[section]
\newtheorem{remark}{\bf Remark}[section]

\newtheorem{lemma}{Lemma}[section]
\newtheorem{corollary}{Corollary}[section]
\newtheorem{example}{Example}[section]

\newtheoremstyle
{remarkstyle}
{}
{11pt}
{}
{}
{\bfseries}
{:}
{     }
{\thmname{#1} \thmnumber{#2} }

\theoremstyle{remarkstyle}

\begin{document}

\title{Revisiting Toeplitz and Hankel random matrices via $*$-convergence of circulant-type matrices}
\author{Arup Bose\thanks{bosearu@gmail.com, \ Research  
supported by J.C.~Bose National Fellowship, JBR/2023/000023 from Anusandhan National Research Foundation, Govt. of India.}}

\author{Pradeep Vishwakarma\thanks{vishwakarmapr.rs@gmail.com, Research  
supported by National Post Doctoral Fellowship, PDF/2025/000076 from Anusandhan National Research Foundation, Govt. of India.}}
\affil{Theoretical Statistics and Mathematics Unit, Indian Statistical Institute, Kolkata, 700108, West Bengal, India.}
	\date{\today}	
	\maketitle
	\begin{abstract} 
    We establish the joint $*$-convergence of a random circulant matrix and a specific deterministic diagonal matrix. We also show that the empirical spectral distributions of skew-circulant and left skew-circulant random matrices converge weakly a.s.~to complex Gaussian and symmetrized Rayleigh distributions, respectively. 
    
    The $*$-convergence of symmetric Toeplitz and Hankel random matrices is well known. So is the weak convergence of their random spectrum. However, not much is known about the limits. We exploit the connections of circulant, reverse circulant, and left skew-circulants with the Hankel and Toeplitz matrices, to show the $*$-convergence of the random symmetric Toeplitz matrix to the sum of two non-commutative self-adjoint variables, each having a real Gaussian distribution. A similar result holds for the non-symmetric Toeplitz matrix, but the variables are not self-adjoint and have a complex Gaussian distribution. The random Hankel matrix is shown to converge in $*$-distribution to a sum of two self-adjoint variables, each of which has a symmetrized Rayleigh distribution. As a consequence of these results, we also obtain a different proof of the convergence of the empirical spectral distribution of symmetric Toeplitz and Hankel matrices, and a slightly different way of expressing the moments of the limit spectral distribution.\\
     

\end{abstract}

\section{Introduction}\label{intro}
Let $\mathcal{A}$ be a $*$-algebra with identity $1_{\mathcal{A}}$. Let $\varphi$ be a complex linear functional on $\mathcal{A}$ with $\varphi(1_{\mathcal{A}})=1$. It is called a state. 
The pair $(\mathcal{A},\varphi)$ is called a non-commutative $*$-probability space, in short an NCP.
Let $\mathbb{C}\langle x_1,\dots,x_p\rangle$ be the set of complex  monomials in non-commutative indeterminates $x_1, x_1^*, \dots,x_p, x_p^*$.  
For any $\{a_1,\dots,a_p\}\subset\mathcal{A}$, their $*$-distribution, also called the joint distribution,  is the collection of all their joint moments $\mu:\mathbb{C}\langle x_1,\dots,x_p\rangle\rightarrow\mathbb{C}$, given by
$\mu(\{a_1,\dots,a_p\})\coloneqq\varphi(Q(\{a_1,\dots,a_p\}))$, $Q\in\mathbb{C}\langle x_1,\dotsb, x_p\rangle$. 
Any $a\in \mathcal{A}$ is called a variable.  It is called self-adjoint if $a=a^*$. If $a$ is self-adjoint, often its moments $\{\varphi(a^k)\}$ determine a unique probability measure on $\mathbb{R}$. A variable $a$ is called normal if $aa^*=a^*a$. If $a$ is normal, often its moments $\{\varphi(a^ka^{*l})\}$ determine a unique probability measure on $\mathbb{C}$. 

{In particular, a self-adjoint $a$ is called (standard) Gaussian if its moments coincide with those of a (standard) Gaussian probability distribution. Likewise, a non self-adjoint normal variable (element) $a$ is called a (standard) complex Gaussian variable, if its moments $\{\varphi(a^ka^{*l})\}$ coincide with the moments $\{\mathbb{E}(N^kN^{*l})\}$, where $N=X+\iota Y$ is a (standard) complex Gaussian random variable, that is, $X$ and $Y$ are independent Gaussian with mean $0$ and variance $1/2$. If $a$ is a (standard) complex Gaussian variable, then the self-adjoint variable $(a+a^*)/\sqrt{2}$ is a (standard) Gaussian variable. A self-adjoint $a$ is called a (standard) symmetrized Rayleigh variable if $\varphi(a^{2m-1})=0$ and $\varphi(a^{2m})=m!$ for every integer $m \geq 1$. These moments uniquely determine the symmetrized Rayleigh probability density 
\begin{equation*}
    f_R(x)=|x|e^{-x^2}, \ -\infty < x < \infty.
\end{equation*}}
Note the distinction between a variable (which is an element of an NCP) and a random variable (which is a real or complex valued measurable function on some probability space). Finally, variables $a$ and $b$ in an NCP $(\mathcal{A}, \varphi)$ are called independent if they commute and for all integers $k, l, m, n$, 
$\varphi (a^k a^{*l}b^m b^{*n})= \varphi(a^k a^{*l})\varphi(b^m b^{*n})$.

Let $(\mathcal{A}_n,\varphi_n)$, $n\ge1$ and $(\mathcal{A}, \varphi)$ be NCPs.
Sequence of elements $\{a_{n,j}\}_{1\leq j\leq p}$ from $\mathcal{A}_n$ is said to converge in $*$-distribution to $\{a_j\}_{1\leq j\leq p} \subset\mathcal{A}$ if 
$\lim_{n\rightarrow\infty}\varphi_n(Q(\{a_{n,j}\}_{1\leq j\leq p }))=\varphi(Q(\{a_j\}_{1\leq j\leq p}))$ for all $Q\in \mathbb{C}\langle x_1,\dots,x_p\rangle$.

Let $\mathcal{M}_n$ be the collection of $n\times n$ random matrices, all whose entries have finite moments of all orders. Let $\mathrm{Tr}(\cdot)$ and $\mathbb{E}$ denote respectively, trace and expectation. Then, $(\mathcal{M}_n, \varphi_n)$ with $\varphi_n\coloneqq n^{-1}\mathbb{E}\mathrm{Tr}(\cdot)$ is an NCP.
There are many interesting results on $*$-convergence of random matrices. For example, under suitable conditions, independent Wigner matrices converge and are asymptotically free of deterministic matrices as well as some random matrices. For example, see \cite{Adhikari2019} and \cite{Nica2006}. For more details on the $*$-convergence, see \cite{Bose2018, Bose2022}.

A related notion of convergence for random matrices is that of the (random) empirical spectral distribution (ESD). Let $A_n\in\mathcal{M}_n$ with 
eigenvalues $\lambda_1,\lambda_2,\dots,\lambda_n$. Its ESD is defined as $\mu(A_n)\coloneqq n^{-1}\sum_{j=1}^{n}\delta_{\lambda_j}$. 
Weak convergence of ESD (in probability or almost surely (a.s)) to a limiting spectral distribution (LSD) has been studied extensively. Matrices for which convergence results are known include Wigner, Toeplitz, Hankel, Circulant, Sample covariance, etc. See  \cite{Anderson2009, Bose2022}, and references therein for examples and details. 

{ The problem of studying the LSD of a random Toeplitz matrix with independent entries was raised in the review paper \cite{Bai1999}}. Bryc et al. \cite{Bryc2006} and Hammond and Miller \cite{Hammond2005} established the convergence of the ESD for symmetric Toeplitz and Hankel random matrices weakly a.s.~to 
non-random LSDs which are symmetric around $0$. The limits are universal in the sense that they do not depend on the specific distribution of the entries of these matrices. The Toeplitz LSD is known to have a Lebesgue density { (see \cite{Sen2011})}. The Hankel LSD is known to be not unimodal. { For more details on the asymptotic behavior of the eigenvalues of these matrices, see \cite{Bose2008, Chatterjee2009, Kargin2009, Liu2011, Liu2012, Onatski2026, Sen2013}, and references therein}. 
These matrices also converge jointly in $*$-distribution to self-adjoint variables whose $(2m)$th moments are the sums of volumes of some polyhedra inside the unit hypercube of $m+1$ dimensions (see Theorem 10.3.1 of \cite{Bose2022}). No clear picture exists of the behavior of the $*$-limits. 


In Section \ref{sec2}, we investigate the joint $*$-convergence of a random circulant matrix with a deterministic diagonal matrix. For $\eta=e^{\iota\theta/n}$, $\theta\in[-\pi,\pi]$, define an $n\times n$ 
diagonal deterministic matrix $D_n(\theta)=\text{diag}(1,\eta,\dots,\eta^{n-1})$. Let $\{c_i\}_{i\ge0}$ and $\{\tilde{c}_i\}_{i\ge0}$ be independent sequences of independent standard Gaussian random variables, and let $C_n=\text{Circ}(c_0,c_1,\dots,c_{n-1})$ and $\tilde{C}_n=\text{Circ}(\tilde{c}_0,\tilde{c}_1\eta,\dots,\tilde{c}_{n-1}\eta^{n-1})$ be circulant matrices. We show that $\{n^{-1/2}C_n,n^{-1/2}\tilde{C}_n,D_n(\theta)\}$ as elements in $(\mathcal{M}_n,\varphi_n=n^{-1}\mathbb{E}\text{Tr})$ jointly converge in $*$-distribution to variables $\{N,\tilde{N},d(\theta)\}$ in some suitable NCP $(\mathcal{A},\varphi)$, where $N$ and $\tilde{N}$ are independent standard complex Gaussian variables, and the $*$-moments of the variable $d(\theta)$ coincides with those of {$e^{\iota\theta U}$} with $U\sim Uniform(0,1)$. This result extends easily to independent copies of the above circulant matrices, and is used in Section \ref{sec4}. 

In Section \ref{sec3}, we show that the LSDs of skew-circulant and left skew-circulant random matrices are the same as those of, respectively the circulant and reverse circulant, namely 
complex Gaussian and symmetrized Rayleigh. These matrices also converge in $*$-distribution. 
Skew-circulant matrices commute,
and 
 converge in $*$-distribution to 
independent standard complex Gaussian variables. On the other hand, independent left skew-circulant matrices half commute and converge in $*$-distribution to half-independent (for definition, see Section \ref{sec32}) self-adjoint variables with symmetrized Rayleigh distributions. 

Let $T_n$, $T_{s,n}$, and $H_n$ be respectively Toeplitz, symmetric Toeplitz, and Hankel matrices whose entries are independent with mean $0$, variance $1$, and all moments uniformly bounded.   
The $*$-convergence of $n^{-1/2}T_{s,n}$ and $n^{-1/2}H_n$ were established in \cite{Bose2011}, and of multiple independent copies of these matrices in \cite{basuboseganghazra2012}. These proofs can also be adapted to $n^{-1/2}T_n$. When the entries are iid with mean $0$ and variance $1$, the a.s.~LSD of $n^{-1/2}T_{s,n}$ and $n^{-1/2}H_n$  were proved in \cite{Bryc2006} and \cite{Hammond2005}. 
In Sections \ref{sec41} and \ref{sec42}, we take a different route to  prove these results, and also provide additional information on the limits. 

In Section \ref{sec41}, first, we use the fact that $T_n$ is the sum of circulant and skew-circulant matrices to show that $n^{-1/2}T_n$ converges in $*$-distribution to $(N_1+N_2)/\sqrt{2}$, where $N_1$ and $N_2$ are standard complex Gaussian variables.
The dependence relation between $N_1$ and $N_2$ is not identified. 
However, 
their joint moments are sums of  appropriate joint moments of $\{N,\tilde{N}, d(\pi)\}$, where $N$ and $\tilde{N}$ are independent standard complex Gaussian variables and the moments of $d(\pi)$ coincide with those of {$e^{\iota\pi U}$} with $U\sim Uniform(0,1)$. We then use the fact that $T_{s,n}$ is the sum of a non-symmetric Toeplitz matrix and its transpose, and the above $*$-convergence result, 
to show that $n^{-1/2}T_{s,n}$ converges in $*$-distribution to $(Z+\tilde{Z})/\sqrt{2}$, where $Z$ and $\tilde{Z}$ are self-adjoint variables with Gaussian distribution, and their joint moments are the sum of finitely many integrals of type (\ref{cdjm}) (for example, see Remark \ref{rem43} (i)).
Finally, we write $H_n$ as the sum of a reverse circulant and a (symmetric) left skew-circulant, to show that $n^{-1/2}H_n$ converges in $*$-distribution to $(Y_1+Y_2)/\sqrt{2}$, where $Y_1$ and $Y_2$ are self-adjoint variables with symmetrized Rayleigh distribution. 
It remains open to explicitly identify the relation between them. However, their joint moments are the sum of finitely many integrals of type (\ref{cdjm}) (for example, see Remark \ref{rem43} (ii)). 

In Section \ref{sec42}, we deal with the LSD of $n^{-1/2}T_{s,n}$  and $n^{-1/2}H_n$. We use their above $*$-convergence results to show that the a.s.~LSD of these matrices exists when the entries of $T_{s,n}$ and $H_n$ are iid with mean $0$ and variance $1$, thereby giving a different proof of the results of  \cite{Bryc2006} and \cite{Hammond2005}.  The moments of these LSD are given respectively by the moments of $(N+\tilde{N})/\sqrt{2}$ and $(Y_1+Y_2)/\sqrt{2}$, thereby providing different expressions for the moments compared to those given by \cite{Bryc2006} and \cite{Hammond2005}.

\section{$*$-convergence of a random circulant and a deterministic matrix} \label{sec2}
Since we shall be using the NCP $(\mathcal{M}_n, n^{-1}
 \mathbb{E} {\rm Tr})$ we introduce the following assumption.
\vskip5pt

\noindent \textbf{Assumption I}. Entries of the matrix are independent, with zero mean, unit variance, and uniformly bounded moments of any order. 
\vskip5pt

Any sequence $\{c_i\}$ generates  an $n\times n$ circulant matrix, $C_n\coloneqq((c_{j-i+n\, (\text{mod}\ n)}))_{0\leq i,j\leq n-1}$.
The classic book \cite{Davis1979} covers the basic theory of circulant matrices. For results on circulant matrices generated by random variables, see \cite{Bosesaha2018}. 

For example, if the entries $\{c_i\}$ are independent and have zero mean, unit variance, and uniformly bounded third moments, then the ESD of $n^{-1/2}C_n$ converges to the standard complex Gaussian distribution in probability. 
Now suppose $C_n^{(1)},\dots, C_n^{(p)}$ are $p$ independent random circulant matrices whose entries satisfy Assumption I.  
Let $(\mathcal{A},\varphi)$ be the $*$-probability space, where $\mathcal{A}$ is generated by independent standard complex Gaussian random variables $N_1, \dots, N_p$ 
and $\varphi$ is the expectation operator. Then, $n^{-1/2}(C_n^{(1)},\dots, C_n^{(p)})$ jointly converge to $(N_1, \ldots, N_p)$ in $*$-distribution (see \cite{Bose2026}). 
That is, for any monomial $Q\in\mathbb{C}\langle x_1,\dots,x_p\rangle$, $p\in\mathbb{N}$, 
\begin{equation}\label{jcc}
    \lim_{n\rightarrow\infty}\varphi_n(Q(n^{-1/2}C_n^{(1)}, \dots, n^{-1/2}C_n^{(p)})){\rightarrow}\varphi(Q(N_1,\dots,N_p)).
\end{equation}

For $\eta=e^{\iota\pi/n}$, let us consider a circulant matrix $\tilde{C}_n=\text{Circ}(c_0,c_1\eta,\dots,c_{n-1}\eta^{n-1})$ where $\{c_i\}$ satisfy Assumption I. 
Let $\tilde{C}_n^{(1)},\dots, \tilde{C}_n^{(p)}$ be independent circulant matrices of type $\tilde{C}_n$. Then, following the proof of Lemma 5.2 of \cite{Bose2026}, it can be shown that 
\begin{equation}\label{jclimit}
    \lim_{n\rightarrow\infty}\varphi_n(Q(n^{-1/2}\tilde{C}_n^{(1)}, \dots, n^{-1/2}\tilde{C}_n^{(p)})){\rightarrow}\varphi(Q(N_1,\dots,N_p)),
\end{equation}
where $N_1, \ldots , N_p$ are iid complex Gaussian variables and $\varphi$ is the expectation operator. 

The question of joint convergence of deterministic and random matrices is rather interesting. For example, it is known that under suitable conditions, Wigner and deterministic matrices converge jointly, and these collections are free of each other (see \cite{Voiculescu1991}). Similar results are known for iid Gaussian and deterministic matrices (see \cite{Nica2006}), and for elliptic and deterministic matrices (see \cite{Adhikari2019}). 
It is not at all easy to establish joint convergence of other random and deterministic matrices, for example, when we consider 
random circulant matrices, Toeplitz matrices or Hankel matrices. Even though the following theorem concerns a very specific case, it will be quite useful later. Let 
\begin{equation}\label{eqn:defndn}D_n(\theta)=\text{diag}\{d_0,d_1,\dots,d_{n-1}\}, \ \  \text{where} \ \ d_j=\eta^j, 0\leq j\leq n-1, \eta=e^{\iota\theta/n}, \theta\in [-\pi, \pi].\end{equation} The following lemma is immediate: { Its proof is given in the Appendix.}
\begin{lemma}\label{lemma1}
The ESD of $D_n(\theta)$ converges weakly to the distribution of { $(\cos(\theta U), \sin(\theta U))$}, where $U\sim Uniform(0,1)$. { As an element of $(\mathcal{M}_n,n^{-1}\text{Tr}(\cdot))$, $D_n(\theta)$ converges to a variable $d(\theta)$ in a $*$-probability space (generated by $d(\theta)$) with the state $\varphi$ defined via limiting values of $n^{-1}\mathrm{Tr}$ operating on $\mathbb{C}\langle D_n(\theta)\rangle$. 
For any monomial $m$, $\varphi(m(d(\theta)))=\mathbb{E} [m(e^{\iota\theta U})]$.}
\end{lemma}

\begin{theorem}\label{thm21}
     Let $C_n$ be a random circulant matrix with independent standard Gaussian entries, and 
     Let $D_n(\theta)$ be as in (\ref{eqn:defndn}). Then there exists a $*$-probability space $(\mathcal{A}, \varphi)$ so that for any  monomial $Q$,
    \begin{equation*}
        \lim_{n\rightarrow\infty}\varphi_n(Q(n^{-1/2}C_n,D_n(\theta))){\rightarrow} \varphi(Q(N,d(\theta))),
    \end{equation*}
    where the distributions ($*$-moments) of $N$ and $d(\theta)$ are respectively those of a standard complex Gaussian variable and as in Lemma \ref{lemma1} respectively.
\end{theorem}
\begin{remark}\label{rem21} In Theorem \ref{thm21}, the definition of $\varphi$, and hence of the ``dependence'' between $N$ and $d(\theta)$, has not been made explicit. Any joint moment of these variables is defined as the value of the limit of the corresponding moment of matrices $n^{-1/2} C_n$ and $D_n$. This essentially defines $\varphi$. It follows from some lower order moment computations that $N$ and $d(\theta)$ are not 
free, independent or half independent. 
\end{remark}

Before we give the proof, let us discuss some preliminaries.
Let $\mathcal{P}_2(2k)$ be the set of all pair-partitions of $\{1, 2,\dots, 2k\}$.  
For any matrix $A$, $A^*$ denotes its complex conjugate.
Let $\epsilon_1,\dots,\epsilon_p\in\{1,*\}$ and $q_1,\dots,q_p\in\mathbb{N}\cup\{0\}$ for $p\ge1$. Let us write $D_n$ for $D_n(\theta)$. Let 
\begin{equation*}
    D_n^{(r)}(\theta)=D^{\epsilon_1}_n\dots D^{\epsilon_{q_r}}_n,
\end{equation*} 
where 
$D_n^{(r)}=I$ for $q_r=0$. We denote the $(i,j)$th entries of $C_n^{\epsilon_r}$ and $D_n^{(r)}$ by $c_{ij}^{\epsilon_r}$ and $d^{(r)}_{ij}$, respectively. Note that $D_n^{(r)}=D_n^{k_r}$ for some $k_r\in\mathbb{Z}$ for each $r=1,\dots,p$. Due to the traciality of $\varphi_n$, it is enough to consider monomials in $n^{-1/2}C_n, n^{-1/2}C_n^*, D_n$, and $D_n^*$ that are of the form  $n^{-p/2}C_n^{\epsilon_1}D_n^{(1)}C_n^{\epsilon_2}D_n^{(2)}\dots C_n^{\epsilon_p}D_n^{(p)}$, and show that limit of their $\mathbb{E}{\rm Tr}$ exists. This is calculated by splitting the trace into sums over several types of partitions. By following standard arguments (see, for example, the proofs of Theorems 9.1.1 and 10.2.1 of \cite{Bose2022}), \textit{only pair partitions potentially contribute to the limit}. We omit the details. Hence, this limit is $0$ if $p$ is odd. 
\vskip5pt 
 
 First, consider $p=2$. Then,  
    \begin{equation*}        \lim_{n\rightarrow\infty}\varphi_n(n^{-1}C_n^{\epsilon_1}D_n^{(1)} C_n^{\epsilon_{2}}D_n^{(2)})=\lim_{n\rightarrow\infty}\frac{1}{n^2}\sum_{i_1,i_2=0}^{n-1}\mathbb{E}[c_{i_1i_2}^{\epsilon_1}c_{i_2i_1}^{\epsilon_2}]d_{i_2}^{(1)}d_{i_1}^{(2)}=\begin{cases}
       0,\  \epsilon_1=\epsilon_2,\\
       \mathbb{E}d^{(1)}\mathbb{E}d^{(2)},\ \epsilon_1\ne\epsilon_2.
    \end{cases}
    \end{equation*}
Now, consider $p=4$. Then, 
    \begin{equation}\label{mom4cal}        \varphi_n(n^{-2}C_n^{\epsilon_1}D_n^{(1)}C_n^{\epsilon_2}D_n^{(2)}C_n^{\epsilon_3}D_n^{(3)}C_n^{\epsilon_4}D_n^{(4)})=\frac{1}{n^3}\hspace{-0.1in}\sum_{i_1,i_2,i_3,i_4=0}^{n-1}\hspace{-0.2in}\mathbb{E}[c_{i_1i_2}^{\epsilon_1}c_{i_2i_3}^{\epsilon_2}c_{i_3i_4}^{\epsilon_3}c_{i_4i_1}^{\epsilon_4}]d_{i_2}^{(1)}d_{i_3}^{(2)}d_{i_4}^{(3)}d_{i_1}^{(4)}.
    \end{equation}
    As the entries are Gaussian, by using Isserlis' formula (see \cite{Bose2022}) for the moments of products of Gaussian variables, we get
    \begin{align*}
        \mathbb{E}[c_{i_1i_2}^{\epsilon_1}c_{i_2i_3}^{\epsilon_2}c_{i_3i_4}^{\epsilon_3}c_{i_4i_1}^{\epsilon_4}]&=\mathbb{E}[c_{i_1i_2}^{\epsilon_1}c_{i_2i_3}^{\epsilon_2}]\mathbb{E}[c_{i_3i_4}^{\epsilon_3}c_{i_4i_1}^{\epsilon_4}]+\mathbb{E}[c_{i_1i_2}^{\epsilon_1}c_{i_3i_4}^{\epsilon_3}]\mathbb{E}[c_{i_2i_3}^{\epsilon_2}c_{i_4i_1}^{\epsilon_4}]+\mathbb{E}[c_{i_1i_2}^{\epsilon_1}c_{i_4i_1}^{\epsilon_4}]\mathbb{E}[c_{i_2i_3}^{\epsilon_2}c_{i_3i_4}^{\epsilon_3}].
    \end{align*}
      Hence, we have the following two cases:
    
\paragraph{Case I} For $\pi\in\mathcal{P}_2(4)=\{\{(1,2), (3,4)\}, \{(1,3), (2,4)\}, \{(1,4), (2,3)\}\}$, we have $\epsilon_s=\epsilon_t$ for at least one $(s,t)\in\pi$. 

Let us consider the term $\mathbb{E}[c_{i_1i_2}^{\epsilon_1}c_{i_2i_3}^{\epsilon_2}]\mathbb{E}[c_{i_3i_4}^{\epsilon_3}c_{i_4i_1}^{\epsilon_4}]$. If $\epsilon_1=\epsilon_2$ and $\epsilon_3\neq \epsilon_4$, without loss of generality say $\epsilon_3=1$ and $\epsilon_4=*$, then $\mathbb{E}[c_{i_1i_2}^{\epsilon_1}c_{i_2i_3}^{\epsilon_2}]\mathbb{E}[c_{i_3i_4}^{\epsilon_3}c_{i_4i_1}^{\epsilon_4}]$ is non-zero if and only if (iff) $i_2-i_1=i_3-i_2$ (mod $n$) and $i_4-i_3=i_4-i_1$ (mod $n$). So, $i_1=i_3$ and $2(i_2-i_1)=an$ for some $a\in\{-1,0,1\}$, which implies $i_2=an/2+i_1$. Thus,
\begin{equation*}
    \lim_{n\rightarrow\infty}\frac{1}{n^3}\sum_{i_1,i_2,i_3,i_4=0}^{n-1}\mathbb{E}[c_{i_1i_2}^{\epsilon_1}c_{i_2i_3}^{\epsilon_1}]\mathbb{E}[c_{i_3i_4}c_{i_1i_4}]d_{i_2}^{(1)}d_{i_3}^{(2)}d_{i_4}^{(3)}d_{i_1}^{(4)}=\lim_{n\rightarrow\infty}\frac{1}{n^3}\sum_{i_1,i_4}d_{an/2+i_1}^{(1)}d_{i_1}^{(2)}d_{i_4}^{(3)}d_{i_1}^{(4)}=0.
\end{equation*} 
Similarly,
\begin{equation*}
    \lim_{n\rightarrow\infty}\frac{1}{n^3}\sum_{i_1,i_2,i_3,i_4=0}^{n-1}\mathbb{E}[c_{i_1i_2}c_{i_3i_2}]\mathbb{E}[c_{i_3i_4}^{\epsilon_3}c_{i_4i_1}^{\epsilon_3}]d_{i_2}^{(1)}d_{i_3}^{(2)}d_{i_4}^{(3)}d_{i_1}^{(4)}=0.
\end{equation*}

Let $\epsilon_1=\epsilon_4\in\{1,*\}$, $\epsilon_2=1$ and $\epsilon_3=*$. Then, the third term, $\mathbb{E}[c_{i_1i_2}^{\epsilon_1}c_{i_4i_1}^{\epsilon_4}]\mathbb{E}[c_{i_2i_3}^{\epsilon_2}c_{i_3i_4}^{\epsilon_3}]$ is non-zero iff $i_2-i_1=i_1-i_4$ (mod $n$) and $i_2=i_4$ which implies $i_1=i_2-\tilde{a}n/2$ for some $\tilde{a}\in\{-1,0,1\}$. Thus,
\begin{equation*}
    \lim_{n\rightarrow\infty}\frac{1}{n^3}\sum_{i_1,i_2,i_3,i_4=0}^{n-1}\mathbb{E}[c_{i_1i_2}^{\epsilon_1}c_{i_4i_1}^{\epsilon_1}]\mathbb{E}[c_{i_2i_3}c_{i_4i_3}]d_{i_2}^{(1)}d_{i_3}^{(2)}d_{i_4}^{(3)}d_{i_1}^{(4)}=\lim_{n\rightarrow\infty}\frac{1}{n^3}\sum_{i_2,i_3=0}^{n-1}d_{i_2}^{(1)}d_{i_3}^{(2)}d_{i_2}^{(3)}d_{i_2-\tilde{a}n/2}^{(4)}=0.
\end{equation*}

Now, consider the second term $\mathbb{E}[c_{i_1i_2}^{\epsilon_1}c_{i_3i_4}^{\epsilon_3}]\mathbb{E}[c_{i_2i_3}^{\epsilon_2}c_{i_4i_1}^{\epsilon_4}]$, and the case $\epsilon_1=\epsilon_3$, $\epsilon_2=1$ and $\epsilon_4=*$. Then, it is non-zero iff for some $a_1,a_2\in\{-1,0,1\}$, $i_2-i_1-i_4+i_3=a_1n$ and $i_3-i_2-i_4+i_1=a_2n$, which implies
$i_4=i_3-i_2+i_1-a_2n=i_2-i_1+i_3-a_1n$. Thus, $i_2=(a_1-a_2)n/2+i_1$, and
\begin{align*}
    \lim_{n\rightarrow\infty}\frac{1}{n^3}\sum_{i_1,i_2,i_3,i_4=0}^{n-1}\mathbb{E}[c_{i_1i_2}^{\epsilon_1}c_{i_3i_4}^{\epsilon_1}]&\mathbb{E}[c_{i_2i_3}c_{i_1i_4}]d_{i_2}^{(1)}d_{i_3}^{(2)}d_{i_4}^{(3)}d_{i_1}^{(4)}\\
    &=\lim_{n\rightarrow\infty}\frac{1}{n^3}\sum_{i_1,i_3=0}^{n-1}d_{(a_1-a_2)n/2+i_1}^{(1)}d_{i_3}^{(2)}d_{i_3-(a_1+a_2)n/2}^{(3)}d_{i_1}^{(4)}=0.
\end{align*}

Let $\epsilon_1=\epsilon_2=\epsilon_3=\epsilon_4$. Then, $\mathbb{E}[c_{i_1i_2}^{\epsilon_1}c_{i_2i_3}^{\epsilon_2}]\mathbb{E}[c_{i_3i_4}^{\epsilon_3}c_{i_4i_1}^{\epsilon_4}]$ and $\mathbb{E}[c_{i_1i_2}^{\epsilon_1}c_{i_4i_1}^{\epsilon_4}]\mathbb{E}[c_{i_2i_3}^{\epsilon_2}c_{i_3i_4}^{\epsilon_3}]$ are non-zero iff $i_2=(a_1n+i_1+i_3)/2$, $i_4=(a_2n+i_1+i_3)/2$ and $i_1=(i_2+i_4-\tilde{a}_1n)/2$, $(i_2+i_4+\tilde{a}_2n)/2$, respectively, for some $a_1,a_2,\tilde{a}_1,\tilde{a_2}\in\{-1,0,1\}$. Also, $\mathbb{E}[c_{i_1i_2}^{\epsilon_1}c_{i_3i_4}^{\epsilon_3}]\mathbb{E}[c_{i_2i_3}^{\epsilon_2}c_{i_4i_1}^{\epsilon_4}]$ is non-zero iff $i_4=i_2-i_1+i_3-a_1n=a_2n-i_3+i_2+i_1$ and $i_3=i_1+(a_1+a_2)n/2$. Hence, for each term the sum in (\ref{mom4cal}) is of order $n^2$. Therefore, all such terms in (\ref{mom4cal}) tend to $0$ as $n\rightarrow\infty$.

\paragraph{Case II} Let us now assume that for a fix $\pi\in\mathcal{P}_2(4)$, $\epsilon_s\neq \epsilon_t$ for all $(s,t)\in\pi$. Note that 
$
    \mathbb{E}[c_{i_1i_2}^{\epsilon_1}c_{i_2i_3}^{\epsilon_2}c_{i_3i_4}^{\epsilon_3}c_{i_4i_1}^{\epsilon_4}]=\sum_{\pi\in\mathcal{P}_2(4)}\prod_{(s,t)\in\pi}\mathbb{E}[c_{i_si_{s+1}}^{\epsilon_s}c_{i_ti_{t+1}}^{\epsilon_t}]$. 
Then, for $\pi=\{(1,2), (3,4)\}$ and $\{(1,4), (2,3)\}$, $\mathbb{E}[c_{i_1i_2}^{\epsilon_1}c_{i_2i_3}^{\epsilon_2}]\mathbb{E}[c_{i_3i_4}^{\epsilon_3}c_{i_4i_1}^{\epsilon_4}]$ and $\mathbb{E}[c_{i_1i_2}^{\epsilon_1}c_{i_4i_1}^{\epsilon_4}]\mathbb{E}[c_{i_2i_3}^{\epsilon_2}c_{i_3i_4}^{\epsilon_3}]$ are non-zero iff $i_1=i_3$ and $i_2=i_4$, respectively. Also, for $\pi=\{(1,3), (2,4)\}$, $\mathbb{E}[c_{i_1i_2}^{\epsilon_1}c_{i_3i_4}^{\epsilon_3}]\mathbb{E}[c_{i_2i_3}^{\epsilon_2}c_{i_4i_1}^{\epsilon_4}]$ is non-zero iff $i_4=i_1-i_2+i_3$ (mod $n$). In this case, (\ref{mom4cal}) converges to some non-zero limit as $n\rightarrow\infty$. 

\vskip5pt

We further elaborate 
with the following examples:

\begin{example}\label{exmp21}
Let $\epsilon_1=\epsilon_3=1$ and $\epsilon_2=\epsilon_4=*$. Then,
\begin{align*}
    \lim_{n\rightarrow\infty}\frac{1}{n^3}\sum_{i_1,i_2,i_3,i_4=0}^{n-1}\hspace{-0.1in}\mathbb{E}[c_{i_1i_2}^{\epsilon_1}c_{i_3i_4}^{\epsilon_3}]\mathbb{E}[c_{i_2i_3}^{\epsilon_2}c_{i_4i_1}^{\epsilon_4}]d_{i_2}^{(1)}d_{i_3}^{(2)}d_{i_4}^{(3)}d_{i_1}^{(4)}&=0 \ \text{(by arguments in Case I)},\\
\lim_{n\rightarrow\infty}\frac{1}{n^3}\sum_{i_1,i_2,i_3,i_4=0}^{n-1}\hspace{-0.1in}\mathbb{E}[c_{i_1i_2}^{\epsilon_1}c_{i_2i_3}^{\epsilon_2}]\mathbb{E}[c_{i_3i_4}^{\epsilon_3}c_{i_4i_1}^{\epsilon_4}]d_{i_2}^{(1)}d_{i_3}^{(2)}d_{i_4}^{(3)}d_{i_1}^{(4)}&=\lim_{n\rightarrow\infty}\frac{1}{n^3}\sum_{i_1,i_2,i_4=0}^{n-1}d_{i_2}^{(1)}d_{i_1}^{(2)}d_{i_4}^{(3)}d_{i_1}^{(4)}\\
    &=\mathbb{E}[d^{(1)}]\mathbb{E}[d^{(3)}]\mathbb{E}[d^{(2)}d^{(4)}],\\
    \lim_{n\rightarrow\infty}\frac{1}{n^3}\sum_{i_1,i_2,i_3,i_4=0}^{n-1}\hspace{-0.1in}\mathbb{E}[c_{i_1i_2}^{\epsilon_1}c_{i_4i_1}^{\epsilon_4}]\mathbb{E}[c_{i_2i_3}^{\epsilon_2}c_{i_3i_4}^{\epsilon_3}]d_{i_2}^{(1)}d_{i_3}^{(2)}d_{i_4}^{(3)}d_{i_1}^{(4)}&=\lim_{n\rightarrow\infty}\frac{1}{n^3}\sum_{i_1,i_2,i_3=0}^{n-1}d_{i_2}^{(1)}d_{i_3}^{(2)}d_{i_2}^{(3)}d_{i_1}^{(4)}\\
    &=\mathbb{E}[d^{(1)}d^{(3)}]\mathbb{E}[d^{(2)}]\mathbb{E}[d^{(4)}].\\
    \end{align*}
Thus,
\begin{equation*}    \lim_{n\rightarrow\infty}\varphi_n(n^{-2}C_nD_n^{(1)}C_n^{*}D_n^{(2)}C_nD_n^{(3)}C_n^{*}D_n^{(4)})=\mathbb{E}[d^{(1)}]\mathbb{E}[d^{(3)}]\mathbb{E}[d^{(2)}d^{(4)}]+\mathbb{E}[d^{(1)}d^{(3)}]\mathbb{E}[d^{(2)}]\mathbb{E}[d^{(4)}].
\end{equation*}
\end{example}
\begin{example}\label{exmp22}
    If $\epsilon_1=\epsilon_2=1$ and $\epsilon_3=\epsilon_4=*$, then
    \begin{align*}
    \lim_{n\rightarrow\infty}\frac{1}{n^3}\sum_{i_1,i_2,i_3,i_4=0}^{n-1}\mathbb{E}[c_{i_1i_2}^{\epsilon_1}c_{i_2i_3}^{\epsilon_2}]\mathbb{E}[c_{i_3i_4}^{\epsilon_3}c_{i_4i_1}^{\epsilon_4}]d_{i_2}^{(1)}d_{i_3}^{(2)}d_{i_4}^{(3)}d_{i_1}^{(4)}&=0\\
    \lim_{n\rightarrow\infty}\frac{1}{n^3}\sum_{i_1,i_2,i_3,i_4=0}^{n-1}\mathbb{E}[c_{i_1i_2}^{\epsilon_1}c_{i_4i_1}^{\epsilon_4}]\mathbb{E}[c_{i_2i_3}^{\epsilon_2}c_{i_3i_4}^{\epsilon_3}]d_{i_2}^{(1)}d_{i_3}^{(2)}d_{i_4}^{(3)}d_{i_1}^{(4)} &=\mathbb{E}[d^{(1)}d^{(3)}]\mathbb{E}[d^{(2)}]\mathbb{E}[d^{(4)}],
\end{align*}
and
\begin{align*}
    \lim_{n\rightarrow\infty}\frac{1}{n^3}\sum_{i_1,i_2,i_3,i_4=0}^{n-1}&\mathbb{E}[c_{i_1i_2}^{\epsilon_1}c_{i_3i_4}^{\epsilon_3}]\mathbb{E}[c_{i_2i_3}^{\epsilon_2}c_{i_4i_1}^{\epsilon_4}]d_{i_2}^{(1)}d_{i_3}^{(2)}d_{i_4}^{(3)}d_{i_1}^{(4)}\\
    &=\lim_{n\rightarrow\infty}\frac{1}{n^3}\sum_{i_1,i_2,i_3=0}^{n-1}d_{i_2}^{(1)}d_{i_3}^{(2)}d_{i_1-i_2+i_3\,(\text{mod $n$})}^{(3)}d_{i_1}^{(4)}\mathbb{I}\{0\leq i_1-i_2+i_3\leq n\},
\end{align*}
where we have used $i_2-i_1\,(\text{mod $n$})=i_3-i_4\,(\text{mod $n$})$, and $\mathbb{I}$ denotes the indicator.
Note that there exist $k_1,k_2,k_3,k_4\in\mathbb{Z}$ such that $d_j^{(r)}=e^{\iota\theta k_rj/n}$ for each $r=1,2,3,4$. Thus, 
\begin{align*}
    \lim_{n\rightarrow\infty}\frac{1}{n^3}\sum_{i_1,i_2,i_3=0}^{n-1}&d_{i_2}^{(1)}d_{i_3}^{(2)}d_{i_1-i_2+i_3\,(\text{mod $n$})}^{(3)}d_{i_1}^{(4)}\mathbb{I}\{0\leq i_1-i_2+i_3\leq n-1\}\\
    &=\lim_{n\rightarrow\infty}\frac{1}{n^3}\sum_{i_1,i_2,i_3=0}^{n-1}\exp\big(\iota\frac{\theta}{n}(k_1i_2+k_2i_3+k_4i_1+k_3(i_1-i_2+i_3\,(\text{mod $n$})))\big)\\
    &\hspace{7cm}\cdot\mathbb{I}\{0\leq \frac{i_1}{n}-\frac{i_2}{n}+\frac{i_3}{n}\leq 1-\frac{1}{n}\}\\
    &=\lim_{n\rightarrow\infty}\frac{1}{n^3}\sum_{i_1,i_2,i_3=0}^{n-1}\exp\big(\iota\frac{\theta}{n}((k_3+k_4)i_1+(k_1-k_3)i_2+(k_2+k_3)i_3\,(\text{mod $n$}))\big)\\
    &\hspace{7cm}\cdot\mathbb{I}\{0\leq \frac{i_1}{n}-\frac{i_2}{n}+\frac{i_3}{n}\leq 1-\frac{1}{n}\}\\
    &=\int_{[0,1]^3}\exp\big(\iota\theta((k_3+k_4)x_1+(k_1-k_3)x_2+(k_2+k_3)x_3\,(\text{mod $1$}))\big)\\
    &\hspace{6.5cm}\cdot\mathbb{I}\{0\leq x_1-x_2+x_3\leq 1\}\,\mathrm{d}x_1\,\mathrm{d}x_2\,\mathrm{d}x_3.
\end{align*}
\end{example}
Keeping the above discussion in mind, we now give the proof.  
\begin{proof}[{\bf Proof of Theorem \ref{thm21}}]
     For $p\ge1$, we have
    \begin{align*}
        \varphi_n(n^{-p/2}C^{\epsilon_1}_nD^{(1)}_n\dots C^{\epsilon_{p}}_nD^{(p)}_n)&=\frac{1}{n^{1+p/2}}\sum_{i_1,\dots,i_{p}=0}^{n-1}\mathbb{E}(c_{i_1i_2}^{\epsilon_1}d_{i_2}^{(1)}c_{i_2i_3}^{\epsilon_2}d_{i_3}^{(2)}\dots c_{i_pi_1}^{\epsilon_p}d_{i_1}^{(p)})\\
        &=\frac{1}{n^{1+p/2}}\sum_{i_1,\dots,i_{p}=0}^{n-1}\mathbb{E}(c_{i_1i_2}^{\epsilon_1}c_{i_2i_3}^{\epsilon_2}\dots c_{i_pi_1}^{\epsilon_p})d_{i_2}^{(1)}d_{i_3}^{(2)}\dots d_{i_1}^{(p)}.
    \end{align*}
    For odd $p$,  $\mathbb{E}(c_{i_1i_2}^{\epsilon_1}c_{i_2i_3}^{\epsilon_2}\dots c_{i_pi_1}^{\epsilon_p})=0$. 
    Let $p=2m$, $m\ge1$. By Isserlis' formula, 
    \begin{equation}\label{pf11}
        \hspace{-0.1in}\varphi_n(n^{-m}C^{\epsilon_1}_nD^{(1)}_n\dots C^{\epsilon_{2m}}_nD^{(2m)}_n)=\frac{1}{n^{1+m}}\sum_{i_1,\dots,i_{2m}=0}^{n-1}\sum_{\pi\in\mathcal{P}_2(2m)}\prod_{(s,t)\in\pi}\hspace{-0.1in}\mathbb{E}(c_{i_si_{s+1}}^{\epsilon_s}c_{i_ti_{t+1}}^{\epsilon_t})\prod_{r=1}^{2m}d^{(r)}_{i_{r+1}},
    \end{equation}
    where $i_{2m+1}=i_1$. 
Note that
    \begin{equation}\label{cjmom}
        \mathbb{E}(c_{i_si_{s+1}}^{\epsilon_s}c_{i_ti_{t+1}}^{\epsilon_t})=\delta_{\epsilon_s\epsilon_t}\delta_{i_{s+1}-i_{s}-i_{t+1}+i_{t}\,(\text{mod}\,n),0}+(1-\delta_{\epsilon_s\epsilon_t})\delta_{i_{s+1}-i_{s}+i_{t+1}-i_{t}\,(\text{mod}\,n),0},
    \end{equation}
    where $\delta_{i,j}=\mathbb{I}\{i=j\}$ is the Kronecker delta. Let 
    \begin{equation}\label{pf15}
        \alpha_r=i_{r+1}-i_r\,\text{(mod $n$)}
    \end{equation}
    for each $r=1,2,\dots,2m$. Then,
    \begin{equation*}
        \mathbb{E}(c_{i_si_{s+1}}^{\epsilon_s}c_{i_ti_{t+1}}^{\epsilon_t})=\delta_{\epsilon_s\epsilon_t}\delta_{\alpha_s-\alpha_t,0}+(1-\delta_{\epsilon_s\epsilon_t})\delta_{\alpha_s+\alpha_t,0}.
    \end{equation*}
    Thus, for $\pi\in\mathcal{P}_2(2m)$, 
    \begin{align}
        \prod_{(s,t)\in\pi}\mathbb{E}(c_{i_si_{s+1}}^{\epsilon_s}c_{i_ti_{t+1}}^{\epsilon_t})&=\prod_{(s,t)\in\pi}[\delta_{\epsilon_s\epsilon_t}\delta_{\alpha_s-\alpha_t,0}+(1-\delta_{\epsilon_s\epsilon_t})\delta_{\alpha_s+\alpha_t,0}]\nonumber\\
        &=\prod_{(s,t)\in\pi}[\delta_{\epsilon_s\epsilon_{\pi(s)}}\delta_{\alpha_s-\alpha_{\pi(s)},0}+(1-\delta_{\epsilon_s\epsilon_{\pi(s)}})\delta_{\alpha_s+\alpha_{\pi(s)},0}]\nonumber\\
        &=\prod_{r=1}^{2m}[\delta_{\epsilon_r\epsilon_{\pi(r)}}\delta_{\alpha_r-\alpha_{\pi(r)},0}+(1-\delta_{\epsilon_r\epsilon_{\pi(r)}})\delta_{\alpha_r+\alpha_{\pi(r)},0}].\label{pf12}
    \end{align}
    On substituting (\ref{pf12}) in (\ref{pf11}), we get
    \begin{align}        
    \lim_{n\rightarrow\infty}&\varphi_n(n^{-m}C^{\epsilon_1}_nD^{(1)}_n\dots C^{\epsilon_{2m}}_nD^{(2m)}_n)\nonumber\\        &\hspace{-0.2in}=\sum_{\pi\in\mathcal{P}_2(2m)}\lim_{n\rightarrow\infty}\frac{1}{n^{1+m}}\sum_{i_1,\dots,i_{2m}=0}^{n-1}\prod_{r=1}^{2m}[\delta_{\epsilon_r\epsilon_{\pi(r)}}\delta_{\alpha_r-\alpha_{\pi(r)},0}+(1-\delta_{\epsilon_r\epsilon_{\pi(r)}})\delta_{\alpha_r+\alpha_{\pi(r)},0}]d_{i_{r+1}}^{(r)}.\label{pf13}
\end{align}

Let us now define 
\begin{equation*}
    \epsilon_{s,t}=\begin{cases}
        +1,\ \epsilon_s=\epsilon_t,\\
        -1,\ \epsilon_t\neq \epsilon_t.
    \end{cases}
\end{equation*}
Then, for a fix $\pi\in\mathcal{P}_2(2m)$, (\ref{pf12}) is non-zero iff
\begin{equation}\label{pf18}
    \alpha_s=\epsilon_{s,t}\alpha_t\,(\text{mod}\,n)\ \ \text{for each $(s,t)\in\pi$}.
\end{equation}
Also, from (\ref{pf15}), we have
\begin{equation}\label{pf16}
    i_r=i_1+\sum_{k=1}^{r-1}\alpha_k\,(\text{mod $n$})\ \ \text{iff}\ \  \sum_{r=1}^{2m}\alpha_r=0\,(\text{mod $n$}).
\end{equation}
Note that using (\ref{pf18}),
\begin{align}
    \sum_{r=1}^{2m}\alpha_r=\sum_{(s,t)\in\pi}(\alpha_s+\alpha_t)&=\sum_{(s,t)\in\pi}(\alpha_s+\epsilon_{s,t}\alpha_s\,(\text{mod $n$}))\nonumber\\
    &=2\sum_{(s,t)\in\pi:\,\epsilon_{s,t}=+1}\alpha_s\,(\text{mod $n$}).\label{p17}
\end{align}
From (\ref{p17}), it follows that if each $\epsilon_{s,t}=-1$ then the consistency condition in (\ref{pf16}) is satisfied and we have $m$-many independent $\alpha_r$'s, and each $i_r$ can be determine by the  $i_1$ and $\alpha_1,\dots,\alpha_{2m}$. As each of these quantities  
have $O(n)$ possible values, the sum in  (\ref{pf13}) is $O(n^{m+1})$. Further, if $\epsilon_{s,t}=+1$ for at least one $(s,t)\in\pi$, that is, $\epsilon_r=\epsilon_{\pi(r)}$ for at least one $r\in\{1,2,\dots,2m\}$, then from (\ref{p17}),
\begin{equation*}
    \sum_{(s,t)\in\pi:\,\epsilon_{s,t}=+1}\alpha_s\,(\text{mod $n$})=0,
\end{equation*}
and in this case the number of independent $\alpha_r$'s reduces to $m-1$, and as $n \to \infty$,
\begin{equation*}
    \frac{1}{n^{1+m}}\sum_{i_1,\dots,i_{2m}=0}^{n-1}\prod_{r=1}^{2m}[\delta_{\epsilon_r\epsilon_{\pi(r)}}\delta_{\alpha_r-\alpha_{\pi(r)},0}+(1-\delta_{\epsilon_r\epsilon_{\pi(r)}})\delta_{\alpha_r+\alpha_{\pi(r)},0}]d_{i_{r+1}}^{(r)}\to 0.
\end{equation*}

For any monomial $Q=n^{-m}C_n^{\epsilon_1}D_n^{(1)}C_n^{\epsilon_2}D_n^{(2)}\dots C_n^{\epsilon_{2m}}D_n^{(2m)}$, define,
\begin{equation*}
    \Pi_Q(\mathcal{P}_2(2m))\coloneqq\{\pi\in\mathcal{P}_2(2m):\,\text{for each}\,(s,t)\in\pi,\,\epsilon_s\neq \epsilon_t\},
\end{equation*}
as the set of pair partitions corresponding to $Q$. Then, the limit (\ref{pf13}) reduces to
\begin{align}
    \lim_{n\rightarrow\infty}&\varphi_n(n^{-m}C^{\epsilon_1}_nD^{(1)}_n\dots C^{\epsilon_{2m}}_nD^{(2m)}_n)\nonumber\\
    &=\sum_{\pi\in\Pi_Q(\mathcal{P}_2(2m))}\lim_{n\rightarrow\infty}\frac{1}{n^{1+m}}\sum_{i_1,\dots,i_{2m}=0}^{n-1}\prod_{r=1}^{2m}\delta_{\alpha_r+\alpha_{\pi(r)},0}d_{i_{r+1}}^{(r)}\nonumber\\    &=\sum_{\pi\in\Pi_Q(\mathcal{P}_2(2m))}\hspace{-0.1in}\lim_{n\rightarrow\infty}\frac{1}{n^{1+m}}\sum_{i_1,\dots,i_{2m}=0}^{n-1}\prod_{r=1}^{2m}\delta_{\alpha_r+\alpha_{\pi(r)},0}\exp\big(\iota\frac{\theta}{n}(k_1i_1+\dots+k_{2m}i_{2m})\big),\label{pf19}
\end{align}
where we have used $d_j^{(r)}=e^{\iota\theta jk_r/n}$ for some $k_r\in\mathbb{Z}$. Note that (\ref{pf19}) is non-zero iff $\alpha_r=-\alpha_{\pi(r)}\,(\text{mod $n$})$, that is,
\begin{equation}\label{pf110}
    i_{r+1}-i_{r}+i_{\pi(r)+1}-i_{\pi(r)}=a_rn,\ r=1,\dots,2m,
\end{equation}
for some $a_1,\dots,a_{2m}\in\{-1,0,1\}$. We observe that for any given $\pi$, the system of equations (\ref{pf110}) has only $m$-many distinct equations with $2m$-many variables $i_1,\dots,i_{2m}$. As $\pi$ is a pair partition, we can write $(m-1)$-many indices $i_r$'s in terms of the remaining $(m+1)$ indices (called free indices). For a given $\pi\in\Pi_{Q}(\mathcal{P}_2(2m))$, suppose $i_1^{\pi},\dots,i_{m+1}^\pi\in\{i_1,\dots,i_{2m}\}$ are free indices. In particular, if $\pi$ is non-crossing, then any of the $m-1$ non-free index $i_{k'}$, $k'\in\{1,2,\dots,2m\}$
is given by $i_{k'}=i_k^\pi$ 
for some $k\in\{1,\dots,m+1\}$. For crossing $\pi$, any non-free $i_{k'}$ is some linear { combination $l_{k'}\{i_1^\pi,\dots,i_{m+1}^\pi\}$} of the free indices. Note that all free indices also have a similar (trivial) representation. 

Let $NC_2(2m)$ be the set of non-crossing pair partitions of $\{1, \ldots, 2m\}$ and $NC_2^c(2m)$ denote its complement. 
Then, for each $\pi\in\Pi_Q(\mathcal{P}_2(2m))$, there exist constants $k_1^{\pi},\dots,k_{m+1}^\pi\in\mathbb{Z}$ depending on $k_1,\dots,k_{2m}$ such that (\ref{pf19}) can be rewritten as follows:
\begin{align}
    \lim_{n\rightarrow\infty}&\varphi_n(n^{-m}C^{\epsilon_1}_nD^{(1)}_n\dots C^{\epsilon_{2m}}_nD^{(2m)}_n)\nonumber\\ 
    &=\sum_{\pi\in\Pi_Q(\mathcal{P}_2(2m))\cap NC_2(2m)}\lim_{n\rightarrow\infty}\frac{1}{n^{1+m}}\sum_{i_1^\pi,\dots,i_{m+1}^\pi=0}^{n-1}\hspace{-0.1in}\exp\big(\iota\frac{\theta}{n} (k_1^\pi i_1^\pi+\dots+k_{m+1}^\pi i_{m+1}^\pi)\big)\nonumber\\
    &\ \ +\sum_{\pi\in\Pi_Q(\mathcal{P}_2(2m))\cap NC_2^c(2m)}\lim_{n\rightarrow\infty}\frac{1}{n^{1+m}}\sum_{i_1^\pi,\dots,i_{m+1}^\pi=0}^{n-1}\hspace{-0.1in}\exp\big(\iota\frac{\theta}{n} (k_1^\pi i_1^\pi+\dots+k_{m+1}^\pi i_{m+1}^\pi)(\text{mod $n$})\big)\nonumber\\
    &\hspace{7cm}\cdot\prod_{r=1}^{2m}\mathbb{I}\{0\leq l_r\{i_1^\pi,\dots,i_{m+1}^\pi\}\leq n-1\}\nonumber\\
    &=\sum_{\pi\in\Pi_Q(\mathcal{P}_2(2m))\cap NC_2(2m)}\int_{[0, 1]^{m+1}}\hspace{-0.1in}\exp(\iota\theta(k_1^\pi x_1+\dots+k_{m+1}^\pi x_{m+1}))\,\mathrm{d}x_1\dots\,\mathrm{d}x_{m+1}\nonumber\\
    &\ \ {+\sum_{\pi\in\Pi_Q(\mathcal{P}_2(2m))\cap NC_2^c(2m)}
    \int_{[0, 1]^{m+1}}\hspace{-0.1in}\exp(\iota\theta(k_1^\pi x_1+\dots+k_{m+1}^\pi x_{m+1})(\text{mod $1$}))}\nonumber\\
    &\hspace{5cm}\cdot\prod_{r=1}^{2m}\mathbb{I}\{0\leq l_r\{x_1,\dots,x_{m+1}\}\leq 1\}\,\mathrm{d}x_1\dots\,\mathrm{d}x_{m+1}.\label{cdjm} 
\end{align}
This completes the proof of Theorem \ref{thm21}.
\end{proof}

{
\begin{remark}
     Note that for a non-crossing partition, in (\ref{pf110}), exactly $m+1$ of the indices $(i_1,\ldots,i_{2m})$ can be chosen freely, while each of the remaining $m-1$ indices equals one of these free indices. Consequently, every choice of the free indices contributes the same value to the first integral in (\ref{cdjm}). However, for a crossing partition, the remaining $m-1$ indices are instead determined as linear combinations of the free indices. Nevertheless, every choice of the free indices again gives the same value for the second integral in (\ref{cdjm}), due to the presence of the modulo operation.
\end{remark}

\begin{remark}\label{gaussianmom:1}
    Note that for each $m\ge1$, (\ref{cdjm}) is bounded by $\#\Pi_Q(\mathcal{P}_2(2m))$ (cardinality of the set $\Pi_Q(\mathcal{P}_2(2m))$). Since $ \#\Pi_Q(\mathcal{P}_2(2m))\leq \#\mathcal{P}_2(2m)$, and $\#\mathcal{P}_2(2m)$ is the $(2m)$th moment of a standard Gaussian distribution, it follows that the limiting $*$-moments of any monomial of order $2m$ in $C_n$ and $D_n$ is bounded by the $(2m)$th moment of a standard Gaussian distribution. Consequently, the moments of any fixed polynomial in $C_n$ and $D_n$ are bounded by the corresponding moments of a Gaussian distribution with an appropriate (not necessarily unit) variance. 
\end{remark}

}

\vskip5pt

{ The next two results are extensions of Theorem \ref{thm21}, and provide the joint convergence of complex circulant matrices together with $D_n(\theta)$. See Appendix for their proofs. }
\begin{corollary}\label{corl1}
    Let $\tilde{C}_n=\text{Circ}(c_0,c_1\eta,\dots,c_{n-1}\eta^{n-1})$, where $\eta=e^{\iota\pi/n}$ and $\{c_j\}_{j\ge0}$ are iid standard Gaussian. Let $D_n(\theta)$ be 
    as in (\ref{eqn:defndn}).
    Then, 
    $\{n^{-1/2}\tilde{C}_n,D_n\}$ 
    as elements of 
    $(\mathcal{M}_n,\varphi_n=n^{-1}\mathbb{E}\text{Tr})$ converge to $\{N,d(\theta)\}$, 
    where $N$ and $d(\theta)$ are as in Theorem \ref{thm21}. 
\end{corollary}
    
\begin{theorem}\label{thm22}
    Let $C_n^{(j)}, \tilde{C}_n^{(j)}$, $j=1,\dots,k$ be independent random circulant matrices as in Theorem \ref{thm21} and Corollary \ref{corl1}, respectively, so that the involved random variables are iid Gaussian. Let $D_n(\theta)$ 
    be as in (\ref{eqn:defndn}).
    Then, $\{n^{-1/2}C_n^{(j)}, n^{-1/2}\tilde{C}_n^{(j)}, D_n (\theta)\}_{1\leq j\leq k}$ jointly converge to $\{N_j,\tilde{N}_j,d (\theta)\}_{1\leq j\leq k}$ in $*$-distribution, { where  both $\{N_j\}$ and $\{\tilde{N}_j\}$ are collections of independent standard complex Gaussian variables (for definition, see Section \ref{intro}) such that $(N_j,\tilde{N}_j)$'s are not independent, and the moments of $d(\theta)$ coincide with those of $e^{\iota\theta U}$, where 
    $U\sim Uniform (0, 1)$}. 
\end{theorem}
\begin{remark}
    As in Remark \ref{rem21}, the relation between the non self-adjoint Gaussian variables and $d(\theta)$ is left unsaid. Clearly 
    every $(N_j, d(\theta))$ and $(\tilde{N}_j, d(\theta))$ are distributed as $(N, d(\theta))$ given in Theorem \ref{thm21}. It may be observed that if $D_n$ is a general deterministic matrix (which itself converges), it is then not easy to establish its joint convergence with one or more circulant matrices. 
\end{remark}
\begin{example}\label{exmp23}
    Consider the $*$-probability space $(\mathcal{M}_n,\varphi_n=n^{-1}\mathbb{E}\text{Tr})$. Let $\{c_j\}$ and $\{\tilde{c}_j\}$ be independent sequence of iid standard Gaussian variables, and $\eta=e^{\iota\theta/n}$. Let us consider a monomial $n^{-2}C_nD_n\tilde{C}_nD_nC_n^*D_n\tilde{C}_n^*D^*_n$ in $n^{-1/2}C_n=n^{-1/2}\text{Circ}(c_0,c_1,\dots,c_{n-1})$ and $n^{-1/2}\tilde{C}_n=n^{-1/2}\text{Circ}(\tilde{c}_0,\tilde{c}_1\eta,\dots,\tilde{c}_{n-1}\eta^{n-1})$. Then,
    \begin{align*}
        \lim_{n\rightarrow\infty}&\varphi_n (n^{-2}C_nD_n\tilde{C}_nD_nC_n^*D^*_n\tilde{C}_n^*D^*_n)\\
        &=\lim_{n\rightarrow\infty}\frac{1}{n^3}\sum_{i_1,i_2,i_3,i_4=0}^{n-1}\mathbb{E}[c_{i_1i_2}\tilde{c}_{i_2i_3}c_{i_3i_4}^*\tilde{c}_{i_4i_1}^*]\eta^{i_3-i_2\,(\text{mod $n$})}\eta^{i_1-i_4\,(\text{mod $n$})}d_{i_1}^*d_{i_2}d_{i_3}d_{i_4}^*\\
        &=\lim_{n\rightarrow\infty}\frac{1}{n^3}\sum_{i_1,i_2,i_3,i_4=0}^{n-1}\mathbb{E}[c_{i_1,i_2}c_{i_4i_3}]\mathbb{E}[\tilde{c}_{i_2i_3}\tilde{c}_{i_1,i_4}]\eta^{i_3-i_2+i_1-i_4\,(\text{mod $n$})}d_{i_1}^*d_{i_2}d_{i_3}d_{i_4}^*\\
        &=\lim_{n\rightarrow\infty}\frac{1}{n^3}\sum_{i_1,i_2,i_3=0}^{n-1}\eta^{0\,(\text{mod $n$})}d_{i_1}^*d_{i_2}d_{i_3}d_{i_1-i_2+i_3\,(\text{mod $n$})}^*\mathbb{I}\{0\leq i_1-i_2+i_3\leq n-1\}\\
        &=\lim_{n\rightarrow\infty}\frac{1}{n^3}\sum_{i_1,i_2,i_3=0}^{n-1}e^{\iota\theta(2i_2/n-2i_1/n)\,(\text{mod $1$})}\mathbb{I}\{0\leq i_1-i_2+i_3\leq n-1\}\\
        &={\iint_{[0,1]^3}e^{\iota\theta2(x_2-x_1)\,(\text{mod}\,1)}\mathbb{I}\{0\leq x_1-x_2+x_3\leq 1\}\,\mathrm{d}x_1\,\mathrm{d}x_2\,\mathrm{d}x_3},
    \end{align*}
where the third equality follows from $i_2-i_1\,(\text{mod $n$})=i_3-i_4\,(\text{mod $n$})$.   Similarly, for the monomial $n^{-2}C_nD_nC_n^*D_n\tilde{C}_nD_n\tilde{C}_n^*D^*_n$, we have
    \begin{align*}
        \lim_{n\rightarrow\infty}\varphi_n (n^{-2}&C_nD_nC_n^*D_n\tilde{C}_nD_n\tilde{C}_n^*D^*_n)\\
        &=\lim_{n\rightarrow\infty}\frac{1}{n^3}\sum_{i_1,i_2,i_3,i_4=0}^{n-1}\mathbb{E}[c_{i_1i_2}c_{i_2i_3}^*\tilde{c}_{i_3i_4}\tilde{c}_{i_4i_1}^*]\eta^{i_4-i_3\,(\text{mod $n$})}\eta^{i_1-i_4\,(\text{mod $n$})}d_{i_1}^*d_{i_2}d_{i_3}d_{i_4}\\
        &=\lim_{n\rightarrow\infty}\frac{1}{n^3}\sum_{i_1,i_2,i_3,i_4=0}^{n-1}\mathbb{E}[c_{i_1,i_2}c_{i_3i_2}]\mathbb{E}[\tilde{c}_{i_3i_4}\tilde{c}_{i_1,i_4}]\eta^{i_1-i_3\,(\text{mod $n$})}d_{i_1}^*d_{i_2}d_{i_3}d_{i_4}\\
        &=\lim_{n\rightarrow\infty}\frac{1}{n^3}\sum_{i_1,i_2,i_4=0}^{n-1}d_{i_1}^*d_{i_2}d_{i_1}d_{i_4}\\
        &={\iint_{[0,1]^2}e^{\iota\theta (x_1+x_2)}\,\mathrm{d}x_1\,\mathrm{d}x_2=-\frac{(e^{\iota\theta}-1)^2}{\theta^2}.}
    \end{align*}
\end{example}

\section{LSD of skew and left skew-circulant matrices} \label{sec3}
We now study the $*$-convergence of 
skew- and left skew-circulant matrices. We also derive the LSDs of these matrices, for which it is enough to make the following weaker assumption. 
\paragraph{\textbf{Assumption II}} Let $\{s_j\}_{j\ge0}$ be independent real rvs such that $\mathbb{E}s_j=0$, $\mathbb{E}s_j^2=1$ and $\sup_{j}\mathbb{E}|s_j|^{2+\delta}<\infty$ for some $\delta>0$.
\vskip5pt

\subsection{Skew-circulant matrix} \label{sec31}
The skew-circulant matrix $S_n=\text{Scirc}(s_0,s_1,\dots,s_{n-1})$ is defined as:
\begin{equation*}
    S_n\coloneqq\begin{bmatrix}
        s_0&s_1&s_2&\dots&s_{n-1}\\
        -s_{n-1}&s_0&s_1&\dots&s_{n-2}\\
        -s_{n-2}&-s_{n-1}&s_0&\dots&s_{n-3}\\
        \vdots&\vdots&\vdots&\ddots&\vdots\\
        -s_1&-s_2&-s_3&\dots&s_0
    \end{bmatrix}.
\end{equation*}
Let $D_n=\text{diag}(1,\eta,\eta^2,\dots,\eta^{n-1})$, $\eta=e^{\iota\pi/n}$ be an $n\times n$ diagonal matrix and $F_n=((F_{ij}))$ be the discrete Fourier transform matrix of order $n$ with $F_{ij}=e^{\iota2\pi ij/n}$, $0\leq i,j\leq n-1$. Define $\tilde{F}_n=D_nF^*_n$. Then, we have (see \cite{Davis1979})
\begin{equation}\label{scdig}
    S_n=\tilde{F}_n\tilde{\Lambda}_n\tilde{F}^*_n,
\end{equation}
where $\tilde{\Lambda}_n=\text{diag}(\lambda_0,\lambda_1,\dots,\lambda_{n-1})$ contains the eigenvalues of $S_n$. Moreover, these  eigenvalues are given by 
\begin{equation}\label{scegn}
    \lambda_j=\beta_{j}+\iota\gamma_{j},\ j=0,1,\dots,n-1,
\end{equation}
where 
\begin{equation*}
    \beta_{j}=\sum_{r=0}^{n-1}s_r\cos\big(\frac{\pi(2j+1)r}{n}\big)\ \ \text{and}\ \ \gamma_{j}=\sum_{r=0}^{n-1}s_r\sin\big(\frac{\pi(2j+1)r}{n}\big).
\end{equation*} 
In particular, we note that
\begin{equation}\label{digsc}
    S_n=D_n F^*_n\tilde{\Lambda}_nF_nD^*_n=D_n\tilde{C}_nD^*_n,
\end{equation}
where $\tilde{C}_n=\text{Circ}(s_0,s_1\eta,\dots,s_{n-1}\eta^{n-1})$ is a circulant matrix. 
\begin{remark}\label{rem31}
    Let $S_n^{(1)}$ and $S_n^{(2)}$ be skew-circulant matrices. Then they commute. To see this, from (\ref{digsc}), we have    
    \begin{equation*}
S_n^{(1)}S_n^{(2)}=D_n\tilde{C}_n^{(1)}D^*_nD_n\tilde{C}_n^{(2)}D^*_n=D_n\tilde{C}_n^{(1)}\tilde{C}_n^{(2)}D^*_n=D_n\tilde{C}_n^{(2)}D^*_nD_n\tilde{C}_n^{(1)}D^*_n=S_n^{(2)}S_n^{(1)},
    \end{equation*}
    where we have used the facts that $D^*_nD_n=I$, and circulant matrices $\tilde{C}_n^{(i)}$, $i=1,2$ commute. 
\end{remark}
The ESD of $S_n$  (on $\mathbb{R}^2$) is defined as
\begin{equation*}
    F_{S_n}(x,y)\coloneqq\frac{1}{n}\sum_{j=0}^{n-1}\mathbb{I}\{\beta_j\leq x,\,\gamma_j\leq y\},\ (x,y)\in\mathbb{R}^2,
\end{equation*}
where $\mathbb{I}$ is the indicator function.
The LSD (in probability) of a random circulant matrix under Assumption II is the complex Gaussian distribution. 
See Chapter 3 of \cite{Bosesaha2018} for a detailed proof.
The second part of Theorem \ref{thm31} states that $S_n$ has the same asymptotic behaviour.
Since its proof
is similar, 
we omit it.  The proof of the first part is also omitted, since Theorem \ref{thm32} is a more general result. 
\begin{theorem}\label{thm31}
    Let $S_n$ be a skew-circulant matrix. Then the following hold:\vskip5pt

    \noindent (i) If the entries $\{s_j\}$ satisfy Assumption I, then $n^{-1/2}S_n$ as an element of $*$-probability space $(\mathcal{M}_n,n^{-1}\mathbb{E}\text{Tr})$ converges to a 
    standard complex Gaussian variable.\vskip5pt

    \noindent (i) If $\{s_j\}$ satisfy Assumption II, then the ESD of $n^{-1/2}S_n$ converges weakly in probability
    to a bivariate Gaussian distribution with mean $0$, and covariance matrix $\Sigma$
    which is diagonal with entries $1/2$.
    \end{theorem}


The following result gives the joint convergence of independent skew-circulant matrices.
\begin{theorem}\label{thm32}
    Let $S_n^{(j)}$, $j=1,\dots,k\in\mathbb{N}$ be independent skew-circulant matrices whose input sequences satisfy Assumption I. Then, $\{n^{-1/2}S_n^{(j)}\}_{1\leq j\leq k}$ as elements of $(\mathcal{M}_n,\varphi_n=n^{-1}\mathbb{E}\text{Tr})$ jointly converge  
    to independent standard complex Gaussian variables.
\end{theorem}
\begin{proof}
    Note that skew-circulant matrices are simultaneously diagonalizable using the unitary matrix $\tilde{F}_n$ as given in (\ref{scdig}). Moreover, from (\ref{digsc}), it follows that a monomial in $\{n^{-1/2}S_n^{(j)},n^{-1/2}S_n^{*^{(j)}}\}_{1\leq j\leq k}$  has the following form:
    \begin{equation*}
        Q=n^{-\sum_{j=1}^{k}(p_j+q_j)}D_n(\tilde{C}_n^{(1)})^{p_1}(\tilde{C}_n^{(1)^{*}})^{q_1}\dots (\tilde{C}_n^{(k)})^{p_k}(\tilde{C}_n^{(k)^{*}})^{q_k}D_n^*
    \end{equation*}for some $p_j,q_j\in\mathbb{N}\cup\{0\}$ for each $j=1,\dots,k$, where $\{\tilde{C}_n^{(j)}\}_{1\leq j\leq k}$ are independent complex circulant matrices and $D_n$ is the diagonal unitary matrix as defined in Corollary \ref{corl1}. Thus,
    \begin{equation*}
        \varphi_n(Q)=n^{-\sum_{j=1}^{k}(p_j+q_j)}\varphi_n((\tilde{C}_n^{(1)})^{p_1}(\tilde{C}_n^{(1)^{*}})^{q_1}\dots (\tilde{C}_n^{(k)})^{p_k}(\tilde{C}_n^{(k)^{*}})^{q_k}).
    \end{equation*}
    Now, the proof follows from (\ref{jclimit}).
\end{proof}

\subsection{Left skew-circulant matrix} \label{sec32}
An $n\times n$ left skew-circulant matrix $L_n=\text{LScirc}(l_0,l_1,\dots,l_{n-1})$ with real entries from $\{l_j\}_{j\ge0}$  is defined as:
\begin{equation*}
    L_n\coloneqq\begin{bmatrix}
        l_0&l_1&l_2&\dots&l_{n-2}&l_{n-1}\\
        l_1&l_2&l_3&\dots&l_{n-1}&-l_0\\
        l_2&l_3&l_4&\dots&-l_0&-l_1\\
        \vdots&\vdots&\vdots&\ddots&\vdots&\vdots\\
        l_{n-1}&-l_0&-l_1&\dots&-l_{n-3}&-l_{n-2}
    \end{bmatrix}.
\end{equation*}
Note that $L_n$ is a symmetric matrix. Let $J_n$ be the $n\times n$ exchange matrix,
\begin{equation}\label{jn}
    J_n=((\delta_{i,n-1-j}))_{0\leq i,j\leq n-1}\coloneqq\begin{bmatrix}
        0&0&0&\dots&0&1\\
        0&0&0&\dots&1&0\\
        \vdots&\vdots&\vdots&\ddots&\vdots&\vdots\\
        0&1&0&\dots&0&0\\
        1&0&0&\dots&0&0
    \end{bmatrix},
\end{equation}
and $S_n=\text{Scirc}(s_0,s_1,\dots,s_{n-1})$ be an $n\times n$ skew-circulant matrix. We note that $L_n=S_nJ_n$ with $l_j=s_{n-1-j}$ for each $j=0,1,\dots,n-1$. Then, from (\ref{digsc}), we get 
\begin{equation}\label{lscJ}
    L_n=\tilde{F}_n\tilde{\Lambda}_n\tilde{F}_n^*J_n.
\end{equation}
As $L_n$ is a symmetric matrix $L_n^*=J_n\tilde{F}_n\tilde{\Lambda}_n^*\tilde{F}_n^*=L_n$. Thus,
\begin{equation}\label{lscegs}
    L_n^2=\tilde{F}_n\tilde{\Lambda}_n\tilde{\Lambda}_n^*\tilde{F}_n^*.
\end{equation}
From (\ref{lscegs}), it follows that the eigenvalues of $L_n^2$ are $|\lambda_0|^2, |\lambda_1|^2,\dots,|\lambda_{n-1}|^2$, where $\lambda_j$, $j=0,\dots,n-1$ are the eigenvalues of a skew-circulant matrix $S_n=\text{Scirc}(l_{n-1},l_{n-2},\dots,l_0)$. 
Recall that the LSD of the reverse circulant matrix is the symmetrized Rayleigh distribution, see Theorem 8.6.2 of \cite{Bose2022}. We now show that $L_n$ has the same LSD.  \vskip5pt

{ The following result is a restatement of Theorem 7.5.1 of \cite{Bose2022}, which gives the universality of the $*$-limit and the LSD for symmetric matrices. Moreover, using similar arguments as given in the proof of Theorem 7.5.1 of \cite{Bose2022}, this can be extended for the $*$-convergence of a non-symmetric matrix. This allows us, without loss of generality, to adopt a convenient choice of entries for the matrices considered hereon. 
\begin{lemma}\label{lemma:univer}
Let $A_n$, $n\ge1$, be symmetric matrices whose entries satisfy Assumption I, and every entry of $A_n$ is repeated only finitely many times in a column and row. Then, the $*$-limiting moments and the LSD (if it exists) of $n^{-1/2}A_n$ do not depend on the distributions of its entries. Moreover, if $A_n$ is non-symmetric, then the $*$-limit of $n^{-1/2}A_n$ is independent of the distribution of its entries (satisfying Assumption I).
\end{lemma}
}

\begin{theorem}\label{thm33}
    Let $L_n$ be an $n\times n$ random left skew-circulant matrix. Then 
    \vskip5pt

    \noindent (i)  If the entries satisfy Assumption I, then $n^{-1/2}L_n$, as an element of 
    $(\mathcal{M}_n,\varphi_n=n^{-1}\mathbb{E}\text{Tr})$, converges to a self-adjoint variable whose odd moments vanish and $(2m)$th moment equals $m!$ for each $m\ge1$. 
    \vskip5pt

    \noindent (ii)  If the entries are iid with mean $0$ and variance $1$, then the ESD of $n^{-1/2}L_n$ converges weakly a.s.~to the symmetrized Rayleigh distribution. 
 \end{theorem}
\begin{proof} 
  (i) 
   In view of Lemma \ref{lemma:univer}, we may assume that the entries of $L_n$ are iid standard Gaussian.  Moreover, $\lim_{n\rightarrow\infty}\varphi_n((n^{-1/2}L_n)^{2m-1})=0$ for all $m\ge1$.  From (\ref{lscegs}), we have
    \begin{align*}
    \mathbb{E}\text{Tr}(n^{-1/2}L_n)^{2m}=n^{-m}\mathbb{E}\text{Tr}(L_n^2)^m&=n^{-m}\mathbb{E}\text{Tr}(\tilde{F}_n\tilde{\Lambda}_n\tilde{\Lambda}_n^*\tilde{F}_n^*)^m\\
    &=n^{-m}\mathbb{E}\text{Tr}(\tilde{\Lambda}_n\tilde{\Lambda}_n^*)^m\\
    &=n^{-m}(\mathbb{E}\lambda_0^m+\sum_{j=1}^{n-1}\mathbb{E}(\beta_j^2+\gamma_j^2)^m),
    \end{align*}
    where we have used (\ref{scegn}) in the last step. Moreover, $\lambda_0$ is Gaussian with mean $0$ and variance $n$. Also, for each $j=1,\dots,n-1$, $\beta_j$ and $\gamma_j$ are independent Gaussian with mean $0$ and variance $n/2$. So, $\beta_j^2+\gamma_j^2\sim\exp(1/n)$ and $\mathbb{E}(\beta_j^2+\gamma_j^2)^m=m!n^m$. Thus,
    \begin{equation*}
        \lim_{n\rightarrow\infty}n^{-1}\mathbb{E}\text{Tr}(n^{-1/2}L_n)^{2m}=\lim_{n\rightarrow\infty}\frac{\mathbb{E}\lambda_0^m+(n-1)n^mm!}{n^{m+1}}=m!,
    \end{equation*} 
    which coincides with the $(2m)$th moment of the symmetrized Rayleigh distribution. This proves the first part. The second part now follows by standard truncation arguments. See for example Lemma 1.4.1 of \cite{Bose2018}.
    We omit the details.
\end{proof}
\subsubsection{Half-commuting and half-independent variables} We recall the concept of half-commuting and half-independent variables. See \cite{Bose2022} for details. 
Let $(\mathcal{A},\varphi)$ be an NCP.
Variables $\{a_j\}_{j\in\mathcal{J}}\subset\mathcal{A}$ half commute if, for any $i,j,k\in \mathcal{J}$, $a_ia_ja_k=a_ka_ja_i$. If $\{a_j\}_{j\in\mathcal{J}}$ half-commute then $\{a_j^2\}_{j\in\mathcal{J}}$ commute.
Let $a_{j_1}\dots a_{j_k}$, $j_1,\dots,j_k\in\mathcal{J}$ be a monomial in $\{a_j\}_{j\in\mathcal{J}}$. It is called symmetric in $\{a_j\}_{j\in\mathcal{J}}$ if the number of appearances of $a_j$ at odd and even positions is equal for each $j$. Otherwise, it is called non-symmetric.
Half-commuting variables $\{a_j\}_{j\in\mathcal{J}}$ are said to be half-independent if $\{a_j^2\}_{j\in\mathcal{J}}$ are independent and, if for all non-symmetric monomial $Q$ in $\{a_j\}_{j\in\mathcal{J}}$, $\varphi(Q)=0$.
\begin{remark}
    Let $\{L_n^{(j)}\}_{j\ge1}$ be left skew-circulant matrices. Then (\ref{lscJ}) implies, 
$L_n^{(j)}=\tilde{F}_n\tilde{\Lambda}^{(j)}_n\tilde{F}_n^*J_n=J_n\tilde{F}_n\tilde{\Lambda}_n^{*^{(j)}}\tilde{F}_n^*$ for each $j\ge1$. Then, for any $i,j,k$,
\begin{align*}
L_n^{(i)}L_n^{(j)}L_n^{(k)}&=\tilde{F}_n\tilde{\Lambda}^{(i)}_n\tilde{F}_n^*J_n(\tilde{F}_n\tilde{\Lambda}^{(j)}_n\tilde{F}_n^*J_n)^*\tilde{F}_n\tilde{\Lambda}^{(k)}_n\tilde{F}_n^*J_n\\
&=\tilde{F}_n\tilde{\Lambda}^{(i)}_n\tilde{F}_n^*J_nJ_n\tilde{F}_n\tilde{\Lambda}^{*^{(j)}}_n\tilde{F}_n^*\tilde{F}_n\tilde{\Lambda}^{(k)}_n\tilde{F}_n^*J_n\\
    &=\tilde{F}_n\tilde{\Lambda}^{(i)}_n\tilde{\Lambda}^{*^{(j)}}_n\tilde{\Lambda}^{(k)}_n\tilde{F}_n^*J_n\\
    &=\tilde{F}_n\tilde{\Lambda}^{(k)}_n\tilde{\Lambda}^{*^{(j)}}_n\tilde{\Lambda}^{(i)}_n\tilde{F}_n^*J_n\\
    &=\tilde{F}_n\tilde{\Lambda}^{(k)}_n\tilde{F}_n^*J_n(\tilde{F}_n\tilde{\Lambda}^{(j)}_n\tilde{F}_n^*J_n)^*\tilde{F}_n\tilde{\Lambda}^{(i)}_n\tilde{F}_n^*J_n\\
    &=L_n^{(k)}L_n^{(j)}L_n^{(i)},
\end{align*}
where we have used that $J_n^2=I$. Thus, left skew-circulant matrices  half-commute.
\end{remark}
    In \cite{Bose2011}, it is shown that independent reverse circulant matrices half-commute, and converge to half-independent self-adjoint variables with symmetrized Rayleigh distributions with unit variance. Theorem \ref{thm:LSC} can be proved along similar lines. We omit the proof. 
\begin{theorem}\label{thm:LSC}
    Let $\{L_n^{(j)}\}_{1\leq j\leq k}$ be independent left skew-circulant random matrices whose entries satisfy Assumption I. Then, $\{n^{-1/2}L_n^{(j)}\}_{1\leq j\leq k}$ converges in $*$-distribution to half-independent variables with symmetrized Rayleigh distribution with unit variance.  
\end{theorem}
\section{
$*$-convergence and LSD of Toeplitz  and Hankel matrices}\label{sec4}
We consider the symmetric and non-symmetric Toeplitz matrices $n^{-1/2}T_{s,n}$ and $n^{-1/2}T_n$,  as well as the symmetric Hankel matrices $n^{-1/2}H_n$. 
Under Assumption I, the $*$-convergence of several independent $n^{-1/2}T_{s,n}$ and $n^{-1/2}H_n$ has been established in \cite{Bose2011} and \cite{basuboseganghazra2012}. The proofs can be easily adapted to also establish the $*$-convergence of independent copies of $n^{-1/2}T_n$. These proofs proceed by directly showing the convergence of $\mathbb{E}\text{Tr}$. We provide alternative, shorter proofs that use the connections between the above matrices and $C_n$, $S_n$, $L_n$, and the results already proved for these matrices. This also helps to show that the limit variables involve the complex Gaussian and symmetric Rayleigh variables. 

When the entries are iid with mean $0$ and variance $1$, the LSD (in the a.s.~sense) of $n^{-1/2}T_{s,n}$ and $n^{-1/2}H_n$ were established by \cite{Bryc2006} and \cite{Hammond2005}. The most crucial step in their proof is to invoke the above $*$-convergence results under Assumption II to show the convergence of the expected ESD. Then the Borel-Cantelli lemma is used to show a.s.~convergence of the ESD, and, finally, standard truncation techniques are used to relax Assumption II. Our proof of the a.s.~convergence of the LSD for these matrices can now be fashioned out of the new proof of the $*$-convergence mentioned in the previous paragraph. The rest of the argument is standard as explained above. 

\subsection{$*$-convergence}
\label{sec41}
The $n\times n$ Toeplitz matrix with input sequence $\{\tau_j\}_{j\in\mathbb{Z}}$ is generated as $T_n\coloneqq ((\tau_{j-i}))_{0\leq i,j\leq n-1}$.
If $\tau_{-j}=\tau_{j}$ for all $j\in\mathbb{N}$, then the matrix is symmetric. 
The following result provides the $*$-convergence of a random non-symmetric Toeplitz matrix. 

\begin{theorem}\label{thm41}
    Let $T_n$ be a non-symmetric Toeplitz matrix generated by iid random variables that satisfy Assumption I.  
    Then, $n^{-1/2}T_n$ as an element of 
    $(\mathcal{M}_n,\varphi_n=n^{-1}\mathbb{E}\text{Tr}(\cdot))$ converges 
    in $*$-distribution 
    to $(N+\tilde{N})/\sqrt{2}$, where 
    $N$ and $\tilde{N}$ are 
    standard complex Gaussian variables (for definition, see Section \ref{intro}). 
    The joint moments of $N$ and $\tilde{N}$ are sums of finite integrals of type (\ref{cdjm}), and are dominated by some Gaussian moments. 
\end{theorem}

\begin{proof} 
If we carefully follow the proof of Theorem 7.5.1 of \cite{Bose2022} for general symmetric patterned matrices, while calculating limits of $\mathbb{E}\text{Tr}$, only pair partitions potentially contribute, and hence, without loss, we may assume that the entries are Gaussian.

Let $\{c_j\}_{j\ge0}$ and $\{s_j\}_{j\ge0}$ be independent sequences of iid mean zero and variance $1$ Gaussian random variables. Let $C_n=\text{Circ}(c_0,c_1,\dots,c_{n-1})$ and $S_n=\text{Scirc}(s_0,s_1,\dots,s_{n-1})$ be $n\times n$ random circulant and skew-circulant matrices. Then, 
\begin{equation}\label{tpl}
    T_n=\frac{C_n+S_n}{\sqrt{2}}=\dfrac{1}{\sqrt{2}}\begin{bmatrix}
        c_0+s_0&c_1+s_1&c_2+s_2&\dots&c_{n-1}+s_{n-1}\\
        c_{n-1}-s_{n-1}&c_0+s_0&c_1+s_1&\dots&c_{n-2}+s_{n-2}\\
        c_{n-2}-s_{n-2}&c_{n-1}-s_{n-1}&c_0+s_0&\dots&c_{n-3}+s_{n-3}\\
        \vdots&\vdots&\vdots&\ddots&\vdots\\
        c_1-s_1&c_2-s_2&c_3-s_3&\dots&c_0+s_0
    \end{bmatrix}
\end{equation}
is a non-symmetric Toeplitz matrix generated by  
independent standard Gaussian random variables $\tau_0=(c_0+s_0)/\sqrt{2}$, $\tau_j=(c_j+s_j)/\sqrt{2}$ and $\tau_{-j}=(c_{n-j}-s_{n-j})/\sqrt{2}$ for $j=1\dots,n-1$. 

First, it is easy to see that for any monomial $Q$ in $\{n^{-1/2}T_n, n^{-1/2}T_n^*\}$ with unequal numbers of $T_n$ and $T_n^*$, $\lim_{n\rightarrow\infty}\varphi_n(Q)=0$. 
From Theorems \ref{thm21} and \ref{thm31}, 
both $n^{-1/2}C_n$ and $n^{-1/2}S_n$ converge in $*$-distribution to standard complex Gaussian variables. It remains to show that they converge jointly. To see this, from (\ref{digsc}), we have 
$n^{-1/2}S_n=D_n(n^{-1/2}\tilde{C}_n)D_n^*$, where $\tilde{C}_n$ and $D_n$ are as defined in Corollary \ref{corl1}. Thus, any typical monomial in 
$\{n^{-1/2}C_n, n^{-1/2}C_n^*, n^{-1/2}S_n, n^{-1/2}S_n^*\}$ is given as 
    \begin{equation*}
        n^{-\sum_{k=1}^{m}(p_j+p_j'+q_j+q_j')/2}C_n^{p_1}(C_n^*)^{q_1}D_n\tilde{C}_n^{p_1'}(\tilde{C}_n^*)^{q_1'}D_n^*\dots C_n^{p_m}(C_n^*)^{q_m}D_n\tilde{C}_n^{p_m'}(\tilde{C}_n^*)^{q_m'}D_n^*,
    \end{equation*}
    where $p_j,p_j',q_j,q_j'\in\mathbb{N}\cup\{0\}$ for each $j=1,\dots,m\in\mathbb{N}$. Hence, in view of Theorem \ref{thm22}, $\{n^{-1/2}C_n, n^{-1/2}S_n\}$ jointly converge in $*$-distribution to say, $\{N, \tilde{N}\}$. In particular, both $N$ and $\tilde{N}$ are standard Gaussian variables. { Invoking (\ref{tpl}), we conclude that the $*$-limit of $n^{-1/2}T_n$ is $N+\tilde{N}$. Also, any polynomial in $\{T_n, T_n^*\}$ is a polynomial in $\{C_n, C_n^*, \tilde{C}_n, \tilde{C}_n^*, D_n(\pi)\}$. Therefore, using the linearity of the state $n^{-1}\mathbb{E}\mathrm{Tr}$, from Theorem \ref{thm22}, the $*$-limits of $n^{-1/2}T_n$ is a sum of integrals of the form (\ref{cdjm}) because the $*$-limit of any monomial in $\{C_n, C_n^*, \tilde{C}_n, \tilde{C}_n^*, D_n(\pi)\}$ has this form, except that it will involve different coefficients in the linear combinations appearing in the formula. (for example, see the proof of Theorem \ref{thm22}, Eq. \ref{CtildeC:jtcovn}, and Example \ref{exmp23}). Moreover, the Gaussian moment bound follows from Remark \ref{gaussianmom:1}. This completes the proof.}
    \end{proof}
\begin{remark}\label{rem41}
    The dependence structure of $N$ and $\tilde{N}$ in Theorem \ref{thm41} is not known. 
\end{remark}
In view of (\ref{jcc}) and Theorem \ref{thm32}, we have the following result. We omit its proof.
\begin{theorem}\label{thm4.2}
    Let $\{n^{-1/2}T_n^{(j)}\}_{1\leq j\leq k}$ be iid non-symmetric Toeplitz matrices whose entries satisfy Assumption I. They 
    jointly converge in $*$-distribution to $\{(N_j+\tilde{N}_j)/\sqrt{2}\}_{1\leq j\leq k}$, where both $\{N_j\}$ and $\{\tilde{N}_j\}$ are collections of independent 
    standard complex Gaussian variables in an appropriate NCP.
\end{theorem}

We now consider the symmetric Toeplitz matrix and the Hankel matrix. {The following result will be needed. For its proof, see Appendix.
\begin{lemma}\label{lemma:digeq}
    Let $A_n=((a_{i,j}))$ be an $n\times n$ symmetric patterned matrix whose entries satisfy Assumption I. Suppose $n^{-1/2}A_n$ as an element of $(\mathcal{M}_n, n^{-1}\mathbb{E}\mathrm{tr})$ converges in $*$-distribution, and its LSD exists. For a constant $c\neq 0$, $\tilde{A}_n=((\tilde{a}_{i,j}))$ be a matrix such that $\tilde{a}_{i,j}=a_{i,j}$ for $1\leq i\neq j\leq n$ and $\tilde{a}_{i,i}=ca_{i,i}$ for each $1\leq i\leq n$. Then,\vskip5pt

    \noindent (i) the $*$-limit of $n^{-1/2}\tilde{A}_n$ coincides to that of $n^{-1/2}A_n$.\vskip5pt

    \noindent (ii) The LSD of $n^{-1/2}\tilde{A}_n$ is the same as the LSD of $n^{-1/2}A_n$.
\end{lemma}
}

\begin{theorem}\label{thm43}
(i)     Let $T_{s,n}=((x_{|i-j|}))_{0\leq i,j\leq n-1}$ be a symmetric Toeplitz matrix whose entries satisfy Assumption I. Then, $n^{-1/2}T_{s,n}$ converges in $*$-distribution to $(Z+\tilde{Z})/\sqrt{2}$, where $Z$ and $\tilde{Z}$ are self-adjoint standard Gaussian variables (for definition, see Section \ref{intro}). Moreover, the joint moments of  $Z$ and $\tilde{Z}$ are sums of integrals of type (\ref{cdjm}). 
\vskip5pt

\noindent (ii) 
Let $H_n$ be an $n\times n$ Hankel matrix whose entries 
     satisfy Assumption I. Then, $n^{-1/2}H_n$ as an element of $(\mathcal{M}_n,\varphi_n=n^{-1}\mathbb{E}\text{Tr})$ converges in $*$-distribution to a self-adjoint variable $(Y_1+Y_2)/\sqrt{2}$ in some NCP 
     $(\mathcal{A},\varphi)$. Here, $Y_1$ and $Y_2$ are symmetrized standard Rayleigh variables (for definition, see Section \ref{intro}). 
     Their joint moments are sums of finite integrals of the type (\ref{cdjm}) and dominated by some Gaussian moments.
    \end{theorem}

\begin{proof} 
{In view of Lemma \ref{lemma:univer}, for both parts, 
we consider the matrices with standard Gaussian entries only. }
\vskip5pt

\noindent (i)  
 Let $T_n=((\tau_{j-i}))_{0\leq j\leq n-1}$ be a non-symmetric Toeplitz matrix whose entries satisfy Assumption I. Note that $T_{s,n}=(T_n+T_n^*)/\sqrt{2}$ is a symmetric Toeplitz matrix with $x_{|i-j|}=(\tau_{j-i}+\tau_{i-j})/\sqrt{2}$ for $0\leq i,j\leq n-1$. From Theorem \ref{thm41}, it follows that $n^{-1/2}T_n$ converges in $*$-distribution to $(N+\tilde{N})/\sqrt{2}$, where $N$ and $\tilde{N}$ are standard complex Gaussian variables. Thus, $n^{-1/2}T_{s,n}$ converges in $*$-distribution to the self-adjoint variable $(N+\tilde{N}+N^*+\tilde{N}^*)/\sqrt{4}$ in an appropriate $*$-probability space. Since 
    $Z:=(N+N^*)/\sqrt{2}$ and $\tilde{Z}:=(\tilde{N}+\tilde{N}^*)/\sqrt{2}$ 
    are standard Gaussian variables, $n^{-1/2}T_{s,n}$ converges in $*$-distribution to $(Z+\tilde{Z})/\sqrt{2}$. 
     
    {Note that the diagonal entries of $T_{s,n}$ are $\sqrt{2}\tau_0$ whose variance is $2$ not $1$. If we replace the diagonal entries $\sqrt{2}\tau_0$ with $\tau_0$, keeping the other entries fixed, then also the resulting matrix is a symmetric Toeplitz matrix whose entries satisfy Assumption I. Now, in view of Lemma \ref{lemma:digeq} (i), the $*$-convergence follows. Moreover, using similar arguments as in the proof of Theorem \ref{thm41}, the $*$-limits can be written as a finite sum of integrals like (\ref{cdjm}), which are dominated by some appropriate Gaussian moments. This completes the proof of Part (i).}
\vskip5pt

\noindent (ii) 
 Let 
    $T_n=((\tau_{j-i}))_{0\leq i,j\leq n-1}=(C_n+S_n)/\sqrt{2}$ be the non-symmetric Toeplitz matrix as in (\ref{tpl}) and $J_n$ be as in (\ref{jn}). 
        Then, $H_n=((h_{i+j}))=T_nJ_n/\sqrt{2}=(C_n+S_n)J_n/\sqrt{2}$ is our choice of the Hankel matrix with $h_j=\tau_{(n-1)-j}/\sqrt{2}$ for all $0\leq j\leq 2n-2$. 
      
    Note that $C_nJ_n=R_n$ and $S_nJ_n=L_n$ are respectively independent reverse circulant and left skew-circulant random matrices whose entries are iid Gaussian with mean $0$ and variance $1$. From Theorem 10.2.1 of \cite{Bose2022}, it follows that any monomial in $R_n$ and $L_n$ with an odd number of $R_n$ or $L_n$ has zero $*$- limit. As reverse and left skew-circulant matrices are symmetric and $J_n^2=I_n$, in view of (\ref{digsc}), any monomial with an even number of $R_n$ and $L_n$ is a monomial in 
    $C_n$, $D_n$ and $\tilde{C}_n$, which are 
    as defined in Theorem \ref{thm22}. Thus, $n^{-1/2}R_n$ and $n^{-1/2}L_n$ jointly converge in $*$-distribution. 
    Moreover, from the results of \cite{Bose2011}, 
    $n^{-1/2}R_n$ converges in $*$-distribution to a self-adjoint standard symmetrized Rayleigh variable. Thus, the $*$-limit follows by invoking 
    Theorem \ref{thm33}. {Remaining part follows by similar arguments as in the proof of Part (i). }
\end{proof}

\begin{remark}
 It can be checked that $Z$ and $\tilde{Z}$ in Theorem \ref{thm43} (i) do not satisfy any known standard dependence structure such as independence, free independence or half independence.
Similar comment holds for $Y_1$ and $Y_2$ in Theorem \ref{thm43} (ii).
\end{remark}
In view of Theorems \ref{thm22}, \ref{thm4.2} and \ref{thm43}, the following joint $*$-convergence result follows: 
\begin{theorem}\label{tjoint:con}
    (i) Let $\{n^{-1/2}T_{s,n}^{(j)}\}_{1\leq j\leq k}$ be independent symmetric Toeplitz matrices whose entries satisfy Assumption I. Then they converge in $*$-distribution to $\{(Z_j+\tilde{Z}_j)/\sqrt{2}\}_{1\leq j\leq k}$, where both $\{Z_j\}$ and $\{\tilde{Z}_j\}$ are collections of independent self-adjoint 
    standard Gaussian variables in a suitable NCP.
    \vskip5pt

\noindent (ii)   Independent random Hankel matrices $\{n^{-1/2}H_n^{(j)}\}_{1\leq j\leq k}
 $ whose entries satisfy Assumption I, converge in $*$-distribution to 
 $\{(Y_1^{(j)}+Y_2^{(j)})/\sqrt{2}\}_{1\leq j\leq k}$, where both $\{Y_1^{(j)}\}_{1\leq j\leq k}$ and $\{Y_2^{(j)}\}_{1\leq j\leq k}$ are collections of self-adjoint half-independent standard symmetrized Rayleigh variables. 
\end{theorem}
{
\begin{remark}
    Note that variables in Theorem \ref{tjoint:con} (i) both collections $\{Z_j\}$ and $\{\tilde{Z}_j\}$ are independent, but $\{(Z_j,\tilde{Z}_j)\}$ is not a collection of independent variables that reaffirm the fact that independent Toeplitz matrices are not asymptotically independent (see \cite{Bose2018}). Similarly, the independent Hankel matrices are not asymptotically half-independent.
\end{remark}
}
\subsection{LSD of symmetric Toeplitz and Hankel matrices}\label{sec42} 
Let $H_n$ and $T_{s,n}$ be respectively Hankel and symmetric Toeplitz random matrices whose entries are iid with mean $0$ and variance $1$.
By results of \cite{Bryc2006, Hammond2005},  the ESD of $n^{-1/2}H_n$ and $n^{-1/2}T_{s,n}$ converge weakly a.s.~to valid probability laws (which are uniquely determined by their moments). 
We now use Theorem \ref{thm43} to give a quick proof of this result. 

\begin{theorem}\label{thm42}
(i) If the entries are iid with mean $0$ and variance $1$, then the ESD of $n^{-1/2}T_{s,n}$ converges weakly a.s.~to the non-random probability law which is determined by the moments of  $(N+\tilde{N})/\sqrt{2}$, where $N$ and $\tilde{N}$ are as in Theorem \ref{thm43} (i).
 \vskip5pt
 
 \noindent 
  (ii)  If the entries are iid with mean $0$ and variance $1$, then the ESD of $n^{-1/2}H_n$ converges weakly a.s.~to the non-random probability law which is determined by the moments of  $(Y_1+Y_2)/\sqrt{2}$, where $Y_1$ and $Y_2$ are as in Theorem \ref{thm43} (ii).
\end{theorem}
\begin{proof}
    (i) First, suppose that the entries satisfy Assumption I.  Note that there is at least one probability measure whose moments are those of $(N+\tilde{N})/\sqrt{2}$. Since these moments are dominated by some Gaussian moments, they satisfy Carleman's condition, implying the uniqueness of this probability measure. So, the expected ESD converges to this measure. Now, by standard fourth moment arguments (for details, see Theorem 1.4.4 of \cite{Bose2018}), the ESD converges to this distribution a.s. Now, Assumption I can be relaxed by a standard truncation argument, { for example, see Lemma 1.4.1 of \cite{Bose2018}}. We omit the details. 
\vskip5pt

\noindent (ii) This proof is similar to that of part (i), and hence it is omitted. 
\end{proof}

{
\begin{remark}
    In Bryc et al. \cite{Bryc2006}, it was conjectured that the LSDs of both symmetric Toeplitz and Hankel matrices have density. For the Toeplitz case, the result was proved in \cite{Sen2011}. The conjecture in the Hankel case has remained open. From Theorem \ref{thm42} (ii), it is clear that the LSD of Hankel is a distribution of the sum of two self-adjoint variables with a common and absolutely continuous distribution. If the dependence structure between these variables could be identified, it might be useful to obtain more information about their joint moments and, in particular, to prove the existence of a density for the LSD of the Hankel matrix.
\end{remark}
}

\begin{remark}\label{rem43}
We calculate the second and fourth moments of the LSDs of $n^{-1/2}T_{s,n}$ and $n^{-1/2}H_{n}$ using their representation in terms of $C_n$, $L_n$, $R_n$ and $S_n$, along with the $*$-limits $(Z+\tilde{Z})/\sqrt{2}$ and $(Y_1+Y_2)/\sqrt{2}$ as obtained in Theorem \ref{thm43}.

Let $\tilde{C}_n$ and $D_n(\pi)$ be matrices as in Theorem \ref{thm22} and $S_n=D_n\tilde{C}_nD_n^*$ be skew-circulant as in (\ref{digsc}). Let 
$\varphi_n=n^{-1}\mathbb{E}\text{Tr}(\cdot))$.
\vskip5pt

\noindent(i) 
From Theorem \ref{thm43} (i). 
\begin{align*}
    \varphi[(Z+\tilde{Z})/\sqrt{2}]^2=\frac{1}{2}\lim_{n\rightarrow\infty}\frac{1}{n}\varphi_n[(T_n+T_n^*)^2]&=\frac{1}{2}\lim_{n\rightarrow\infty}\frac{1}{n}\varphi_n[T_nT_n^*+T_n^*T_n]\\
    &=\frac{1}{2}\lim_{n\rightarrow\infty}\frac{1}{n}\varphi_n[C_nC_n^*+S_nS_n^*]\\
    &=\frac{1}{2}\lim_{n\rightarrow\infty}\frac{1}{n^2}\mathbb{E}\text{Trace}[C_nC_n^*+S_nS_n^*]=1,
\end{align*}
where we have used $\lim_{n\rightarrow\infty}\frac{1}{n^2}\mathbb{E}\text{Trace}[C_nC_n^*]=\lim_{n\rightarrow\infty}\frac{1}{n^2}\mathbb{E}\text{Trace}[S_nS_n^*]=1$. Also, we have used the fact that the limit moments of $n^{-1}T_n^2$ and $n^{-1}{T_n^*}^{2}$ are $0$, and the penultimate step follows from  (\ref{tpl}) and Theorem \ref{thm21}. 

The fourth moment is
\begin{align}
    \varphi[&{(Z+\tilde{Z})}/{\sqrt{2}}]^4\nonumber\\
    &=\frac{1}{4}\lim_{n\rightarrow\infty}\frac{1}{n^2}\varphi_n[T_n^2{T_n^*}^{2}+{T_n^*}^{2}T_n^2+T_nT_n^*T_nT_n^*+T_n^*T_n^2T_n^*+T_n{T_n^*}^2T_n+T_n^*T_nT_n^*T_n]\nonumber\\
    &=\frac{1}{4}\lim_{n\rightarrow\infty}\frac{1}{n^2}(4\varphi_n[T_n^2{T_n^*}^{2}]+2\varphi_n[T_nT_n^*T_nT_n^*]).\label{tmom}
\end{align}
Now,
\begin{align*}
    \lim_{n\rightarrow\infty}\frac{1}{n^2}\varphi_n[T_n^2{T_n^*}^{2}]&=\frac{1}{4}\lim_{n\rightarrow\infty}\frac{1}{n^2}\varphi_n[C_n^2{C_n^*}^2+S_n^2{S_n^*}^2+C_nS_nC_n^*S_n^*\\
    &\ \ +S_nC_nC_n^*S_n^*+C_nS_nS_n^*C_n^*+S_nC_nS_n^*C_n^*]\\
    &=\frac{1}{4}\lim_{n\rightarrow\infty}\frac{1}{n^2}(\varphi_n[C_n^2{C_n^*}^2]+\varphi_n[S_n^2{S_n^*}^2]+\varphi_n[C_nS_nC_n^*S_n^*]\\
    &\ \ +2\varphi_n[C_nC_n^*S_nS_n^*]+\varphi_n[S_nC_nS_n^*C_n^*]),
\end{align*}
where we have used the traciality of $\varphi$ and the commuting property of  circulant and skew-circulant matrices. From (\ref{jcc}) and Theorem \ref{thm31}, we have $\lim_{n\rightarrow\infty}\frac{1}{n^2}\varphi_n[C_n^2{C_n^*}^2]=2=\lim_{n\rightarrow\infty}\frac{1}{n^2}\varphi_n[S_n^2{S_n^*}^2]$, and
\begin{align*}
    \lim_{n\rightarrow\infty}\frac{1}{n^2}\varphi_n[C_nC_n^*S_nS_n^*]&=\lim_{n\rightarrow\infty}\frac{1}{n^3}\mathbb{E}\text{Trace}[C_nC_n^*D_n(\pi)\tilde{C}_n\tilde{C}_n^*D_n^*(\pi)]\\
    &=\lim_{n\rightarrow\infty}\frac{1}{n^3}\sum_{i_1,i_2,i_3,i_4=0}^{n-1}\mathbb{E}[c_{i_1i_2}c_{i_2i_3}^*]\mathbb{E}[\tilde{c}_{i_3i_4}\tilde{c}_{i_4i_1}^*]e^{\iota\frac{\pi}{n}(i_4-i_3\,(\text{mod $n$}))}\\
    &\hspace{5cm}  \cdot e^{\iota\frac{\pi}{n}(i_1-i_4\,(\text{mod $n$}))}d_{i_3}(\pi)d_{i_1}^*(\pi)\\
    &=\lim_{n\rightarrow\infty}\frac{1}{n^3}\sum_{i_1,i_2,i_4=0}^{n-1}d_{i_1}d_{i_1}^*=1,\\
    \lim_{n\rightarrow\infty}\frac{1}{n^2}\varphi_n[C_nS_nC_n^*S_n^*]&=\lim_{n\rightarrow\infty}\frac{1}{n^2}\varphi_n[C_nD_n(\pi)\tilde{C}_nD_n^*(\pi)C_n^*D_n(\pi)\tilde{C}_n^*D_n^*(\pi)]\\
    &=\lim_{n\rightarrow\infty}\frac{1}{n^3}\sum_{i_1,i_2,i_3,i_4=0}^{n-1}\hspace{-0.1in}\mathbb{E}[c_{i_1i_2}c_{i_3i_4}^*]\mathbb{E}[\tilde{c}_{i_2i_3}\tilde{c}_{i_4i_1}^*]e^{\iota\frac{\pi}{n}(i_3-i_2\,(\text{mod $n$}))}\\
    &\hspace{2cm} \cdot e^{\iota\frac{\pi}{n}(i_1-i_4\,(\text{mod $n$}))}d_{i_2}(\pi)d_{i_3}^*(\pi)d_{i_4}(\pi)d_{i_1}^*(\pi)\\
    &=\lim_{n\rightarrow\infty}\frac{1}{n^3}\sum_{i_1,i_2,i_3=0}^{n-1}\mathbb{I}\{0\leq i_1-i_2+i_3\leq n-1\}\\
    &=\lim_{n\rightarrow\infty}\frac{1}{n^3}\frac{2n^3+n}{3}=\frac{2}{3},
\end{align*} 
where we have used $i_2-i_1\,(\text{mod $n$})=i_3-i_4\,(\text{mod $n$})$ to get the penultimate step, and the last equality follows from a standard counting argument. We omit its details. Similarly, we get $\lim_{n\rightarrow\infty}\frac{1}{n^2}\varphi_n[S_nC_nS_n^*C_n^*]=2/3$. Thus, 
\begin{equation}\label{tmom1}
    \lim_{n\rightarrow\infty}\frac{1}{n^2}\varphi_n[T_n^2{T_n^*}^{2}]=\frac{1}{4}(6+\frac{4}{3})=\frac{11}{6}.
\end{equation}
Moreover,
\begin{equation}\label{tmom2}
    \lim_{n\rightarrow\infty}\frac{1}{n^2}\varphi_n[T_nT_n^*T_nT_n^*]=\frac{1}{4}\lim_{n\rightarrow\infty}\frac{1}{n^2}(\varphi_n[C_nC_n^*]^2+\varphi_n[S_nS_n^*]^2+4\varphi_n[C_nC_n^*S_nS_n^*])=\frac{5}{3}.
\end{equation}
On substituting (\ref{tmom1}) and (\ref{tmom2}) in (\ref{tmom}), we get $\varphi[(Z+\tilde{Z})/\sqrt{2}]^4=8/3$, which agrees with the fourth moment of the LSD of a symmetric Toeplitz matrix, see page 46 of \cite{Bose2018}.
\vskip5pt
\noindent (ii) 
We now calculate the second and fourth moments of the LSD of $n^{-1/2}H_n$. 
Let $R_n=C_nJ_n$ and $L_n=S_nJ_n$ be 
as in the proof of Theorem \ref{thm43} (ii). Then, $R_n^2=(C_nJ_n)(C_nJ_n)^*=C_nC_n^*$ and using (\ref{digsc}),
\begin{equation*}
    L_n^2=(S_nJ_n)(S_nJ_n)^*=S_nS_n^*=D_n(\pi)\tilde{C}_n\tilde{C}_n^*D_n^*(\pi).
\end{equation*}

The second moment of $(Y_1+Y_2)/\sqrt{2}$ is  given by
\begin{equation*}
   \varphi[(Y_1+Y_2)/\sqrt{2}]^{2}=\frac{1}{2}[\varphi(Y_1^2)+\varphi(Y_2^2)+\varphi(Y_1Y_2)+\varphi(Y_2Y_1)]=1,
\end{equation*}
where we have used $\varphi(Y_1^2)=\varphi(Y_2^2)=1$ and  
\begin{equation*}    \varphi(Y_1Y_2)=\varphi(Y_2Y_1)=\lim_{n\rightarrow\infty}\varphi_n(n^{-1}R_nL_n)=0.
\end{equation*}
Its fourth moment is
\begin{align}
 \varphi[(Y_1+Y_2)/\sqrt{2}]^{4}&=\frac{1}{4}[\varphi(Y_1^4)+\varphi(Y_2^4)+\varphi(Y_1^2Y_2^2)+\varphi(Y_2^2Y_1^2)+\varphi(Y_1Y_2Y_1Y_2)\nonumber\\
 &\ \ +\varphi(Y_1Y_2^2Y_1)+\varphi(Y_2Y_1Y_2Y_1)+\varphi(Y_2Y_1^2Y_2)]\nonumber\\
 &=\frac{1}{4}[\varphi(Y_1^4)+\varphi(Y_2^4)+4\varphi(Y_1^2Y_2^2)+2\varphi(Y_1Y_2Y_1Y_2)],\label{smom}
\end{align}
where we have used the traciality of the state $\varphi$, and the fact that moments of monomials with odd numbers of $Y_1$ and $Y_2$ are zero. As $Y_1$ and $Y_2$ are symmetrized Rayleigh variables, we have $\varphi(Y_1^4)=\varphi(Y_2^4)=2$. From the proof of Theorem \ref{thm42} (ii), we get
\begin{align}
\varphi(Y_1^2Y_2^2)&=\lim_{n\rightarrow\infty}\varphi_n(n^{-2}R_n^2L_n^2)\nonumber\\
&=\lim_{n\rightarrow\infty}\varphi_n(n^{-2}C_nC_n^*D_n(\pi)\tilde{C}_n\tilde{C}_n^*D_n^*(\pi))\nonumber\\
    &=\lim_{n\rightarrow\infty}\frac{1}{n^3}\sum_{i_1,i_2,i_3,i_4=0}^{n-1}\mathbb{E}[c_{i_1i_2}c_{i_2i_3}^*\tilde{c}_{i_3i_4}\tilde{c}_{i_4i_1}^*]e^{\iota\frac{\pi}{n}(i_4-i_3\,(\text{mod $n$}))}\nonumber\\
    &\hspace{4cm}\cdot e^{-\iota\frac{\pi}{n}(i_4-i_1\,(\text{mod $n$}))}d_{i_3}(\pi)d_{i_1}^*(\pi)\nonumber\\
    &=\lim_{n\rightarrow\infty}\frac{1}{n^3}\sum_{i_1,i_2,i_3,i_4=0}^{n-1}\mathbb{E}[c_{i_1i_2}c_{i_3i_2}]\mathbb{E}[\tilde{c}_{i_3i_4}\tilde{c}_{i_1i_4}]e^{\iota\frac{\pi}{n}(i_1-i_3\,(\text{mod $n$}))}d_{i_3}(\pi)d_{i_1}^*(\pi)\nonumber\\
    &=\lim_{n\rightarrow\infty}\frac{1}{n^3}\sum_{i_1,i_2,i_4=0}^{n-1}d_{i_1}(\pi)d_{i_1}^*(\pi)=1.\label{smom1}
\end{align}
Also, 
\begin{equation}\label{smom2}    \varphi(Y_1Y_2Y_1Y_2)=\lim_{n\rightarrow\infty}\varphi_n(n^{-2}R_nL_nR_nL_n)=\lim_{n\rightarrow\infty}\varphi_n(n^{-2}C_nS_n^*C_nS_n^*)=0.
\end{equation}
Thus, on substituting (\ref{smom1}) and (\ref{smom2}) in (\ref{smom}) yields $\varphi[(Y_1+Y_2)/\sqrt{2}]^{4}=2$, which agrees with the fourth moment of the LSD of Hankel matrix given on page 46 of \cite{Bose2018}.
\end{remark}

{
\appendix
\section{Proof of some results}	
\begin{proof}[{\bf Proof of Lemma \ref{lemma1}}]
    The ESD of $D_n(\theta)$ is given as
    \begin{equation*}
        F_{D_n}(x,y)=\frac{1}{n}\sum_{j=0}^{n-1}\mathbb{I}\Big\{\cos(\frac{j\theta}{n})\leq x,\ \sin(\frac{j\theta}{n})\leq y\Big\},\ (x,y)\in\mathbb{R}^2.
    \end{equation*}
    Then,
    \begin{align*}
        \lim_{n\rightarrow\infty}F_{D_n}(x,y)&=\int_{0}^{1}\mathbb{I}\{\cos(\theta u)\leq x,\ \sin(\theta u)\leq y\}\,\mathrm{d}u=\mathbb{P}\{\cos(\theta U)\leq x,\ \sin(\theta U)\leq y\},
    \end{align*}
    where $U\sim Uniform (0,1)$.

    For the $*$-convergence, we have
    \begin{align*}
        \lim_{n\rightarrow\infty}\frac{1}{n}\mathrm{Tr}(D_n(\theta)^l{D_n^*(\theta)}^m)&=\lim_{n\rightarrow\infty}\frac{1}{n}\sum_{j=0}^{n-1}\exp(\iota\frac{\theta j(l-m)}{n}),\ l,m\ge0\\
        &=\lim_{n\rightarrow\infty}\frac{1-\exp(\iota\theta(l-m))}{n(1-\exp(\iota\frac{\theta(l-m)}{n}))}\\
        &=(1-\exp(\iota\theta(l-m)))\frac{\iota}{\theta(l-m)}\\
        &=\int_{0}^{1}e^{\iota\theta(l-m) u}\,\mathrm{d}u=\mathbb{E}[e^{\iota (l-m)\theta U}].
    \end{align*}
    This completes the proof.
\end{proof}

\begin{proof}[{\bf Sketch of the proof of Corollary \ref{corl1}}] We only provide the outline of the proof.\\

\noindent {\bf Step I.} The eigenvalues of the circulant matrix $n^{-1/2}\tilde{C}_n=n^{-1/2}Circ(c_0,c_1\eta,\dots,c_{n-1}\eta^{n-1})$ with $\eta=e^{\iota\pi/n}$, $n\ge1$ are given as $  \tilde{\lambda}_j=\tilde{\alpha}_j+\iota\tilde{\beta}_j$, $j=0,\dots,n-1$, where 
\begin{equation*}
    \tilde{\alpha}_j=\frac{1}{\sqrt{n}}\sum_{k=0}^{n-1}c_k\cos\Big((2j+1)\frac{\pi k}{n}\Big),\ \ \tilde{\beta}_j=\frac{1}{\sqrt{n}}\sum_{k=0}^{n-1}c_k\sin\Big((2j+1)\frac{\pi k}{n}\Big).
\end{equation*}
Then, $\mathbb{E}[\tilde{\alpha}_j]=\mathbb{E}[\tilde{\beta}_j]=0$, as $\mathbb{E}[c_k]=0$ for each $k\ge0$. Now, using the independence of $c_k$'s and $\mathbb{E}[c_k^2]=1$, for $j\neq \frac{n-1}{2}$, we have
\begin{align*}
    \mathbb{E}[\tilde{\alpha}_j^2]&=\frac{1}{n}\sum_{k=0}^{n-1}\cos^2\Big((2j+1)\frac{\pi k}{n}\Big)\\
    &=\frac{1}{n}\sum_{k=0}^{n-1}\frac{\cos\Big((2j+1)\frac{2\pi k}{n}\Big)+1}{2}=\frac{1}{2}.
\end{align*}
Similarly, $\mathbb{E}[\tilde{\beta}_j^2]=1/2$. Also,
\begin{align*}
    \mathbb{C}\mathrm{ov}(\tilde{\alpha}_j,\tilde{\beta}_j)&=\frac{1}{n}\sum_{k=0}^{n-1}\cos\Big((2j+1)\frac{\pi k}{n}\Big)\sin\Big((2j+1)\frac{\pi k}{n}\Big)\\
    &=\frac{1}{2n}\sum_{k=0}^{n-1}\sin\Big((2j+1)\frac{2\pi k}{n}\Big)=0.
\end{align*}
Now, following the steps of the proof of Theorem 3.2.1 of \cite{Bosesaha2018}, it can be shown that the ESD (for definition see \cite{Bose2022}, p. 108) of  $n^{-1/2}\tilde{C}_n$ converges in $L^2$ to a bivariate Gaussian distribution whose marginals are independent and both have variance $1/2$.

For $*$-convergence, using the commutativity and diagonalizability of circulant matrices, any typical $*$-moment of $n^{-1/2}\tilde{C}_n$, as an element of $(\mathcal{M}_n,n^{-1}\mathbb{E}\mathrm{Tr}(\cdot))$ is given by
\begin{equation}\label{trace:ctilde}
    \frac{1}{n}\mathbb{E}\mathrm{Tr}((n^{-1/2}\tilde{C}_n)^l(n^{-1/2}\tilde{C}^*_n)^m)=\frac{1}{n}\sum_{k=0}^{n-1}\mathbb{E}[\tilde{\lambda}_k^l\Bar{\tilde{\lambda}}_k^m],\ l,m\ge0
\end{equation}
where $\Bar{\tilde{\lambda}}_k$ denotes the complex conjugate of $\tilde{\lambda}_k$. As the input sequence $\{c_k\}_{k\ge0}$ satisfies Assumption I, any power of $\tilde{\lambda}_j$ and $\Bar{\tilde{\lambda}}_j$ are uniformly integrable. Now, using the convergence of ESD $\mathbb{E}[\tilde{\lambda}_k^l\Bar{\tilde{\lambda}}_k^m]\rightarrow\mathbb{E}[\tilde{N}^l\Bar{\tilde{N}}^m]$ for each $l,m\ge0$, where $\tilde{N}$ is a standard complex Gaussian variable. Hence, the average (\ref{trace:ctilde}) converges to $\mathbb{E}[\tilde{N}^l\Bar{\tilde{N}}^m]$ as $n\rightarrow\infty$. Thus, $n^{-1/2}\tilde{C}_n$ converges in $*$-distribution to $\tilde{N}$.

\vskip5pt

\noindent {\bf Step II.} Keeping the notations of the proof of Theorem \ref{thm21}, a typical monomial in $n^{-1/2}\tilde{C}_n$, $n^{-1/2}\tilde{C}_n^*$ and $D_n$ is given by 
$n^{-p/2}\tilde{C}_n^{\epsilon_1}D_n^{(1)}\tilde{C}_n^{\epsilon_2}D_n^{(2)}\dots \tilde{C}_n^{\epsilon_p}D_n^{(p)}$, and 
\begin{align*}
    \varphi_n(&n^{-p/2}\tilde{C}_n^{\epsilon_1}D_n^{(1)}\tilde{C}_n^{\epsilon_2}D_n^{(2)}\dots \tilde{C}_n^{\epsilon_p}D_n^{(p)})\\
    &=\frac{1}{n^{1+p/2}}\sum_{i_1,\dots,i_{p}=0}^{n-1}\mathbb{E}(\tilde{c}_{i_1i_2}^{\epsilon_1}d_{i_2}^{(1)}\tilde{c}_{i_2i_3}^{\epsilon_2}d_{i_3}^{(2)}\dots \tilde{c}_{i_pi_1}^{\epsilon_p}d_{i_1}^{(p)})\\
        &=\frac{1}{n^{1+p/2}}\sum_{i_1,\dots,i_{p}=0}^{n-1}\mathbb{E}(c_{i_1i_2}^{\epsilon_1}c_{i_2i_3}^{\epsilon_2}\dots c_{i_pi_1}^{\epsilon_p})(\eta^{\epsilon_1})^{i_2-i_1\,(\text{mod}\,n)}\dots (\eta^{\epsilon_p})^{i_1-i_p\,(\text{mod}\,n)} d_{i_2}^{(1)}d_{i_3}^{(2)}\dots d_{i_1}^{(p)}.
\end{align*}
Therefore, using similar arguments as in the proof of Theorem \ref{thm21}, we conclude that $(\tilde{C}_n, D_n(\theta))$ jointly converges in $*$-distribution. Indeed, from Step I, it follows that the $*$-limit is $(\tilde{N},d(\theta))$.
\end{proof}

\vskip3pt
\begin{proof}[{\bf Sketch of the proof of Theorem \ref{thm22}}]
    Recall the notations of the proof of Theorem \ref{thm21}. For $p\geq 1$ and $t_1,\dots,t_p\in\{1,\dots,k\}$, 
    let $A_n^{(t_1)},\dots,A_n^{(t_p)}\in\{C_n^{(j)}, \tilde{C}_n^{(j)},\ j=1,\dots,k\}$. Then, a monomial in $n^{-1/2}C_n^{(j)}, {n^{-1/2}}\tilde{C}_n^{(j)}$, $j=1,\dots,k$ and $D_n(\theta)$ is given as
    \begin{equation*}
        Q\coloneqq n^{-p/2}({A_n^{(t_1)}})^{\epsilon_1}D_n^{(1)}({A_n^{(t_2)}})^{\epsilon_2}D_n^{(2)}\dots (A_n^{(t_{p-1})})^{\epsilon_{p-1}}D_n^{(p-1)}(A_n^{(t_p)})^{\epsilon_p}D_n^{(p)}.
    \end{equation*}
Then,
\begin{align*}
    \varphi_n(Q)=\frac{1}{n^{1+p/2}}\sum_{i_1,\dots,i_p=0}^{n-1}\mathbb{E}[({a_{i_1i_2}^{(t_1)}})^{\epsilon_1}({a_{i_2i_3}^{(t_2)}})^{\epsilon_2}\dots (a_{i_{p-1}i_p}^{(t_{p-1})})^{\epsilon_{p-1}}(a_{i_pi_1}^{(t_p)})^{\epsilon_p}]d_{i_2}^{(1)}d_{i_3}^{(2)}\dots d_{i_1}^{(p)}
\end{align*}
As the entries are independent Gaussian random variables, for odd $p$, 
\begin{equation*}
    \mathbb{E}[({a_{i_1i_2}^{(t_1)}})^{\epsilon_1}({a_{i_2i_3}^{(t_2)}})^{\epsilon_2}\dots (a_{i_{p-1}i_p}^{(t_{p-1})})^{\epsilon_{p-1}}(a_{i_pi_1}^{(t_p)})^{\epsilon_p}]=0.
\end{equation*}
Note that $\tilde{C}^{(j)}$ is generated by $(\tilde{c}_0^{(j)},\tilde{c}_1^{(j)}\eta,\dots,\tilde{c}_{n-1}^{(j)}\eta^{n-1})$, where $\eta=e^{i\pi/n}$. For even $p=2m$, we have
\begin{multline*}
    ({a_{i_1i_2}^{(t_1)}})^{\epsilon_1}({a_{i_2i_3}^{(t_2)}})^{\epsilon_2}\dots (a_{i_{{2m}-1}i_{2m}}^{(t_{{2m}-1})})^{\epsilon_{{2m}-1}}(a_{i_{2m}i_1}^{(t_{2m})})^{\epsilon_{2m}}\\=({b_{i_1i_2}^{(t_1)}})^{\epsilon_1}({b_{i_2i_3}^{(t_2)}})^{\epsilon_2}\dots (b_{i_{{2m}-1}i_{2m}}^{(t_{{2m}-1})})^{\epsilon_{{2m}-1}}(b_{i_{2m}i_1}^{(t_{2m})})^{\epsilon_{2m}}\eta^{l(i_1,\dots,i_{2m})\,(\text{mod}\,n)},
\end{multline*}
where $b_{i,j}^{(t)}$'s are standard Gaussian random variables (from $\{c_i^{(j)},\ \tilde{c}_i^{(j)},\ j=1,2\dots,k\}_{i\ge0}$), and $l(i_1,\dots,i_{2m})$ is a linear combination of $i_1,\dots,i_{2m}$ (for example, see Example \ref{exmp23}). Now, by Isserlis' formula
\begin{equation}\label{CtildeC:jtcovn}
     \varphi_n(Q)=\frac{1}{n^{1+m}}\sum_{i_1,\dots,i_{2m}=0}^{n-1}\sum_{\pi\in\mathcal{P}_2(2m)}\prod_{(r,s)\in\pi}\mathbb{E}[({b_{i_ri_{r+1}}^{(t_r)}})^{\epsilon_r}(b_{i_{s}i_{s+1}}^{(t_{s})})^{\epsilon_{s}}]\eta^{l(i_1,\dots,i_{2m})\,(\text{mod}\,n)} d_{i_2}^{(1)}d_{i_3}^{(2)}\dots d_{i_1}^{(2m)}.
\end{equation}
Clearly, for $t_s\neq t_r$, we have $\mathbb{E}[({a_{i_ri_{r+1}}^{(t_r)}})^{\epsilon_r}(a_{i_{s}i_{s+1}}^{(t_{s})})^{\epsilon_{s}}]=0$. Therefore, only those $\pi$ contribute for which $t_r=t_s$ for each $(r,s)\in\pi$ (i.e. we only have indexed pair matching). Now, the rest of the proof follows similar lines to that of Theorem \ref{thm21}. Moreover, from the $*$-convergence of circulant matrices and Step II of the proof of Corollary \ref{corl1}, the joint $*$-limit is $\{N^{(j)}, \tilde{N}^{(j)}, d(\theta),\ j=1,\dots,k\}$, as mentioned in  Theorem \ref{thm22}.
\end{proof}

\begin{proof}[{\bf Proof of Lemma \ref{lemma:digeq}}] (i) It is sufficient to show that all the $*$-moments of $n^{-1/2}(\tilde{A}_n-A_n)$ converge to $0$ as $n\rightarrow\infty$. For $m\ge1$, we have
\begin{align*}
    \lim_{n\rightarrow\infty}\frac{1}{n}\mathbb{E}\mathrm{Tr}[(n^{-1/2}(\tilde{A}_n-A_n))^m]&=\lim_{n\rightarrow\infty}\frac{1}{n^{1+m/2}}\sum_{i=1}^{n}(c-1)^m\mathbb{E}[a_{i,i}^m]\\
    &\leq \lim_{n\rightarrow\infty}\frac{1}{n^{m/2}}(c-1)^mK_m=0,
\end{align*}
where we have used $\sup_{i}\mathbb{E}|a_{i,i}|^m<K_m<\infty$, due to Assumption I.

For Part (ii), the following result will be needed (see Lemma 1.3.2 of \cite{Bose2018}):
\begin{lemma}\label{lsd:dist}
    Let $A$ and $B$ be $n\times n$ symmetric random matrices with LSD $F_A$ and $F_B$, respectively. Then,
    \begin{equation*}
        W_2^2(F_A,F_B)\leq \frac{1}{n}||A-B||_F^2,
    \end{equation*}
    where $W_2$ is the $2$-Wasserstein distance, and $||\cdot||_F$ is the Frobenius norm.
\end{lemma}

Suppose $F_{A_n}$ and $F_{\tilde{A}_n}$ denote the LSDs of $n^{-1/2}A_n$ and $n^{-1/2}\tilde{A}_n$, respectively. Then,
\begin{equation*}
    W_2^2(F_{\tilde{A}_n}, F_{A_n})\leq \frac{1}{n^2}\mathrm{Tr}[(\tilde{A}_n-A_n)^2]=\frac{1}{n^2}\sum_{i=1}^{n}(c-1)^2a_{i,i}^2.
\end{equation*}
Hence,
\begin{equation*}
    W_2^4(F_{\tilde{A}_n}, F_{A_n})\leq \frac{1}{n^4}(c-1)^4\sum_{i,j=1}^{n}a_{i,i}^2a_{j,j}^2.
\end{equation*}
From Assumption I, there exists a constant $0<K<\infty$ such that
\begin{equation*}
    \mathbb{E}[W_2^4(F_{\tilde{A}_n}, F_{A_n})]\leq \frac{(c-1)^4}{n^4}\Big[n+(n^2-n)K\Big]=\frac{(c-1)^4(1+(n-1)K)}{n^3},
\end{equation*}
which yields
\begin{equation*}
    \sum_{n=1}^{\infty}\mathbb{E}[W_2^4(F_{\tilde{A}_n}, F_{A_n})]<\infty.
\end{equation*}
Therefore, using the first Borel-Cantelli lemma, $W_2(F_{\tilde{A}_n}, F_{A_n})\rightarrow0$ a.s., and consequently, the LSDs of $n^{-1/2}\tilde{A}_n$ and $n^{-1/2}A_n$ are a.s. equal. 
\end{proof}
}	

\begin{thebibliography}{00}	
\bibitem{Anderson2009}
Anderson, G.W., Guionnet, A. and Zeitouni, O. (2009). \textit{An Introduction to Random Matrices}, Cambridge University Press, Cambridge.
\bibitem{Adhikari2019}
Adhikari, K. and Bose, A. (2019). Brown measure and asymptotic freeness of elliptic and related matrices. \textit{Random Matrix: Theory Appl.}, \textbf{8}(2), 1950007.
{ \bibitem{Bai1999}
Bai, Z.D. (1999). Methodologies in spectral analysis of large-dimensional random matrices: a review. \textit{Statist. Sinica} \textbf{9}, 611-677.
\bibitem{Bose2008}
Bose, A. and Sen, A. (2008). Another look at the moment method for large dimensional random matrices. \textit{Electron. J. Probab.} \textbf{13}, 588-628.}
\bibitem{Bose2011}
Bose, A., Hazra, R.S. and Saha, K. (2011). Convergence of joint moments for independent random patterned matrices. \textit{Ann. Probab.}, \textbf{39}(4), 1607-1620.
\bibitem{basuboseganghazra2012}
Basu, R., Bose,  A., Ganguly, S. and
Hazra, R.S. (2012).  Joint convergence of several copies of
different patterned random matrices.
\textit{Electron. J. Probab.}, \textbf{17}(82), 1-33.
\bibitem{Bosesaha2018} Bose, A. and Saha, K. (2018). \textit {Random Circulant Matrices}, Chapman \& Hall. 
\bibitem{Bose2018}
Bose, A. (2018). \textit{Patterned Random Matrices}, Chapman \& Hall.
\bibitem{Bryc2006}
Bryc, W., Dembo, A. and Jiang, T. (2006).  Spectral measure of large random Hankel, Markov and Toeplitz matrices. \textit{Ann. Probab.}, \textbf{34}(1), 1-38.
\bibitem{Bose2022}
Bose, A. (2022).\textit{ Random Matrices and Non-Commutative Probability},  Chapman \& Hall.
\bibitem{Bose2026}
Bose, A. and Vishwakarma, P. (2026). Patterned matrices with random walk entries. arXiv:2512.04612v2
{ \bibitem{Chatterjee2009}
Chatterjee, S. (2009). Fluctuations of eigenvalues and second order Poincar\'e inequalities. \textit{Probab. Theory Related Fields} \textbf{143}, 1-40.}
\bibitem{Davis1979}
Davis,  P.J. (1979). \textit{Circulant Matrices}, John Wiley \& Sons.
\bibitem{Hammond2005}
 Hammond, C. and Miller, S.J. (2005). Distribution of eigenvalues for the ensemble of real symmetric Toeplitz matrices. \textit{J. Theor. Probab.}, \textbf{18}(3), 537-566.

{
\bibitem{Kargin2009}
Kargin, V. (2009). Spectrum of random Toeplitz matrices with band structure. \textit{Electron. Commun. Probab.} \textbf{14}, 412-421.
\bibitem{Liu2011}
Liu, D.Z. and Wang, Z.D. (2011). Limit distribution of eigenvalues for random Hankel and Toeplitz band matrices. \textit{J. Theor. Probab.} \textbf{24}, 988-1001.
\bibitem{Liu2012}
Liu, D.Z., Sun, X. and Wang, Z. (2012). Fluctuations of eigenvalues for random Toeplitz and related matrices. \textit{Electron. J. Probab.} \textbf{17}, 1-22.
}
\bibitem{Nica2006}
Nica, A. and Speicher, R. (2006). \textit{Lectures on the Combinatorics of Free Probability}, Cambridge University Press.

{
\bibitem{Onatski2026} Onatski, A. and Kargin, V. (2026). The smallest singular value of large random rectangular Toeplitz and circulant matrices. \textit{Electron. J. Probab.} \textbf{31}, 1-30.
\bibitem{Sen2011}
Sen, A. and Vi\'rag, B. (2011). Absolute continuity of the limiting eigenvalue distribution of the random Toeplitz matrix.
\textit{Electron. Commun. Probab.} \textbf{16}, 706-711.
\bibitem{Sen2013}
Sen, A. and Vi\'rag, B. (2013). The top eigenvalue of the random Toeplitz matrix and the sine kernel. \textit{Ann. Probab.} \textbf{41}(6), 4050-4079.
}

\bibitem{Voiculescu1991}
Voiculescu, D. (1991). Limit laws for random matrices and free products. \textit{Invent. Math.}, \textbf{104}, 201-220.
\end{thebibliography}
\end{document}